\documentclass[12pt, reqno]{amsart}
\usepackage{amssymb}
\usepackage{eucal}
\usepackage{amsmath}
\usepackage{amscd}
\usepackage[dvips]{color}
\usepackage{multicol}
\usepackage[all]{xy}           
\usepackage{graphicx}
\usepackage{color}
\usepackage{colordvi}
\usepackage{xspace}
\usepackage{tikz}

\usepackage{ifpdf}
\ifpdf
 \usepackage[colorlinks,final,backref=page,hyperindex]{hyperref}
\else
 \usepackage[colorlinks,final,backref=page,hyperindex,hypertex]{hyperref}
\fi

\begin{document}
\newcommand {\emptycomment}[1]{} 

\newcommand{\nc}{\newcommand}
\newcommand{\delete}[1]{}
\nc{\mfootnote}[1]{\footnote{#1}} 
\nc{\todo}[1]{\tred{To do:} #1}

\nc{\mlabel}[1]{\label{#1}}  
\nc{\mcite}[1]{\cite{#1}}  
\nc{\mref}[1]{\ref{#1}}  
\nc{\meqref}[1]{\eqref{#1}} 
\nc{\mbibitem}[1]{\bibitem{#1}} 

\delete{
\nc{\mlabel}[1]{\label{#1}  
{\hfill \hspace{1cm}{\bf{{\ }\hfill(#1)}}}}
\nc{\mcite}[1]{\cite{#1}{{\bf{{\ }(#1)}}}}  
\nc{\mref}[1]{\ref{#1}{{\bf{{\ }(#1)}}}}  
\nc{\meqref}[1]{\eqref{#1}{{\bf{{\ }(#1)}}}} 
\nc{\mbibitem}[1]{\bibitem[\bf #1]{#1}} 
}

\newtheorem{thm}{Theorem}[section]
\newtheorem{lem}[thm]{Lemma}
\newtheorem{cor}[thm]{Corollary}
\newtheorem{pro}[thm]{Proposition}
\newtheorem{conj}[thm]{Conjecture}
\theoremstyle{definition}
\newtheorem{defi}[thm]{Definition}
\newtheorem{ex}[thm]{Example}
\newtheorem{rmk}[thm]{Remark}
\newtheorem{pdef}[thm]{Proposition-Definition}
\newtheorem{condition}[thm]{Condition}

\renewcommand{\labelenumi}{{\rm(\alph{enumi})}}
\renewcommand{\theenumi}{\alph{enumi}}

\nc{\tred}[1]{\textcolor{red}{#1}}
\nc{\tblue}[1]{\textcolor{blue}{#1}}
\nc{\tgreen}[1]{\textcolor{green}{#1}}
\nc{\tpurple}[1]{\textcolor{purple}{#1}}
\nc{\btred}[1]{\textcolor{red}{\bf #1}}
\nc{\btblue}[1]{\textcolor{blue}{\bf #1}}
\nc{\btgreen}[1]{\textcolor{green}{\bf #1}}
\nc{\btpurple}[1]{\textcolor{purple}{\bf #1}}

\nc{\rp}[1]{\textcolor{blue}{Ruipu:#1}}
\nc{\cm}[1]{\textcolor{red}{Chengming:#1}}
\nc{\li}[1]{\textcolor{blue}{#1}}
\nc{\lir}[1]{\textcolor{blue}{Li:#1}}


\nc{\twovec}[2]{\left(\begin{array}{c} #1 \\ #2\end{array} \right )}
\nc{\threevec}[3]{\left(\begin{array}{c} #1 \\ #2 \\ #3 \end{array}\right )}
\nc{\twomatrix}[4]{\left(\begin{array}{cc} #1 & #2\\ #3 & #4 \end{array} \right)}
\nc{\threematrix}[9]{{\left(\begin{matrix} #1 & #2 & #3\\ #4 & #5 & #6 \\ #7 & #8 & #9 \end{matrix} \right)}}
\nc{\twodet}[4]{\left|\begin{array}{cc} #1 & #2\\ #3 & #4 \end{array} \right|}

\nc{\rk}{\mathrm{r}}
\newcommand{\g}{\mathfrak g}
\newcommand{\h}{\mathfrak h}
\newcommand{\pf}{\noindent{$Proof$.}\ }
\newcommand{\frkg}{\mathfrak g}
\newcommand{\frkh}{\mathfrak h}
\newcommand{\Id}{\rm{Id}}
\newcommand{\gl}{\mathfrak {gl}}
\newcommand{\ad}{\mathrm{ad}}
\newcommand{\add}{\frka\frkd}
\newcommand{\frka}{\mathfrak a}
\newcommand{\frkb}{\mathfrak b}
\newcommand{\frkc}{\mathfrak c}
\newcommand{\frkd}{\mathfrak d}
\newcommand {\comment}[1]{{\marginpar{*}\scriptsize\textbf{Comments:} #1}}

\nc{\tforall}{\text{ for all }}

\nc{\svec}[2]{{\tiny\left(\begin{matrix}#1\\
#2\end{matrix}\right)\,}}  
\nc{\ssvec}[2]{{\tiny\left(\begin{matrix}#1\\
#2\end{matrix}\right)\,}} 

\nc{\typeI}{local cocycle $3$-Lie bialgebra\xspace}
\nc{\typeIs}{local cocycle $3$-Lie bialgebras\xspace}
\nc{\typeII}{double construction $3$-Lie bialgebra\xspace}
\nc{\typeIIs}{double construction $3$-Lie bialgebras\xspace}

\nc{\bia}{{$\mathcal{P}$-bimodule ${\bf k}$-algebra}\xspace}
\nc{\bias}{{$\mathcal{P}$-bimodule ${\bf k}$-algebras}\xspace}

\nc{\rmi}{{\mathrm{I}}}
\nc{\rmii}{{\mathrm{II}}}
\nc{\rmiii}{{\mathrm{III}}}
\nc{\pr}{{\mathrm{pr}}}
\newcommand{\huaA}{\mathcal{A}}

\nc{\mcdot}{{}}

\nc{\OT}{constant $\theta$-}
\nc{\T}{$\theta$-}
\nc{\IT}{inverse $\theta$-}


\nc{\asi}{ASI\xspace}
\nc{\dualp}{transposed Poisson\xspace}
\nc{\Dualp}{Transposed Poisson\xspace}
\nc{\dualpop}{{\bf TPois}\xspace}
\nc{\ldualp}{derivation-transposed Poisson\xspace}

\nc{\spdualp}{sp-dual Poisson \xspace} \nc{\aybe}{AYBE\xspace}

\nc{\admset}{\{\pm x\}\cup K^{\times} x^{-1}}

\nc{\dualrep}{gives a dual representation\xspace}
\nc{\admt}{admissible to\xspace}

\nc{\ciri}{\circ_{\rm I}}
\nc{\cirii}{\circ_{\rm II}}
\nc{\ciriii}{\circ_{\rm III}}

\nc{\opa}{\cdot_A}
\nc{\opb}{\cdot_B}

\nc{\post}{positive type\xspace}
\nc{\negt}{negative type\xspace}
\nc{\invt}{inverse type\xspace}

\nc{\pll}{\beta}
\nc{\plc}{\epsilon}

\nc{\ass}{{\mathit{Ass}}}
\nc{\comm}{{\mathit{Comm}}}
\nc{\dend}{{\mathit{Dend}}}
\nc{\zinb}{{\mathit{Zinb}}}
\nc{\tdend}{{\mathit{TDend}}}
\nc{\prelie}{{\mathit{preLie}}}
\nc{\postlie}{{\mathit{PostLie}}}
\nc{\quado}{{\mathit{Quad}}}
\nc{\octo}{{\mathit{Octo}}}
\nc{\ldend}{{\mathit{ldend}}}
\nc{\lquad}{{\mathit{LQuad}}}

 \nc{\adec}{\check{;}} \nc{\aop}{\alpha}
\nc{\dftimes}{\widetilde{\otimes}} \nc{\dfl}{\succ} \nc{\dfr}{\prec}
\nc{\dfc}{\circ} \nc{\dfb}{\bullet} \nc{\dft}{\star}
\nc{\dfcf}{{\mathbf k}} \nc{\apr}{\ast} \nc{\spr}{\cdot}
\nc{\twopr}{\circ} \nc{\tspr}{\star} \nc{\sempr}{\ast}
\nc{\disp}[1]{\displaystyle{#1}}
\nc{\bin}[2]{ (_{\stackrel{\scs{#1}}{\scs{#2}}})}  
\nc{\binc}[2]{ \left (\!\! \begin{array}{c} \scs{#1}\\
    \scs{#2} \end{array}\!\! \right )}  
\nc{\bincc}[2]{  \left ( {\scs{#1} \atop
    \vspace{-.5cm}\scs{#2}} \right )}  
\nc{\sarray}[2]{\begin{array}{c}#1 \vspace{.1cm}\\ \hline
    \vspace{-.35cm} \\ #2 \end{array}}
\nc{\bs}{\bar{S}} \nc{\dcup}{\stackrel{\bullet}{\cup}}
\nc{\dbigcup}{\stackrel{\bullet}{\bigcup}} \nc{\etree}{\big |}
\nc{\la}{\longrightarrow} \nc{\fe}{\'{e}} \nc{\rar}{\rightarrow}
\nc{\dar}{\downarrow} \nc{\dap}[1]{\downarrow
\rlap{$\scriptstyle{#1}$}} \nc{\uap}[1]{\uparrow
\rlap{$\scriptstyle{#1}$}} \nc{\defeq}{\stackrel{\rm def}{=}}
\nc{\dis}[1]{\displaystyle{#1}} \nc{\dotcup}{\,
\displaystyle{\bigcup^\bullet}\ } \nc{\sdotcup}{\tiny{
\displaystyle{\bigcup^\bullet}\ }} \nc{\hcm}{\ \hat{,}\ }
\nc{\hcirc}{\hat{\circ}} \nc{\hts}{\hat{\shpr}}
\nc{\lts}{\stackrel{\leftarrow}{\shpr}}
\nc{\rts}{\stackrel{\rightarrow}{\shpr}} \nc{\lleft}{[}
\nc{\lright}{]} \nc{\uni}[1]{\tilde{#1}} \nc{\wor}[1]{\check{#1}}
\nc{\free}[1]{\bar{#1}} \nc{\den}[1]{\check{#1}} \nc{\lrpa}{\wr}
\nc{\curlyl}{\left \{ \begin{array}{c} {} \\ {} \end{array}
    \right .  \!\!\!\!\!\!\!}
\nc{\curlyr}{ \!\!\!\!\!\!\!
    \left . \begin{array}{c} {} \\ {} \end{array}
    \right \} }
\nc{\leaf}{\ell}       
\nc{\longmid}{\left | \begin{array}{c} {} \\ {} \end{array}
    \right . \!\!\!\!\!\!\!}
\nc{\ot}{\otimes} \nc{\sot}{{\scriptstyle{\ot}}}
\nc{\otm}{\overline{\ot}}
\nc{\ora}[1]{\stackrel{#1}{\rar}}
\nc{\ola}[1]{\stackrel{#1}{\la}}
\nc{\pltree}{\calt^\pl}
\nc{\epltree}{\calt^{\pl,\NC}}
\nc{\rbpltree}{\calt^r}
\nc{\scs}[1]{\scriptstyle{#1}} \nc{\mrm}[1]{{\rm #1}}
\nc{\dirlim}{\displaystyle{\lim_{\longrightarrow}}\,}
\nc{\invlim}{\displaystyle{\lim_{\longleftarrow}}\,}
\nc{\mvp}{\vspace{0.5cm}} \nc{\svp}{\vspace{2cm}}
\nc{\vp}{\vspace{8cm}} \nc{\proofbegin}{\noindent{\bf Proof: }}
\nc{\proofend}{$\blacksquare$ \vspace{0.5cm}}
\nc{\freerbpl}{{F^{\mathrm RBPL}}}
\nc{\sha}{{\mbox{\cyr X}}}  
\nc{\ncsha}{{\mbox{\cyr X}^{\mathrm NC}}} \nc{\ncshao}{{\mbox{\cyr
X}^{\mathrm NC,\,0}}}
\nc{\shpr}{\diamond}    
\nc{\shprm}{\overline{\diamond}}    
\nc{\shpro}{\diamond^0}    
\nc{\shprr}{\diamond^r}     
\nc{\shpra}{\overline{\diamond}^r}
\nc{\shpru}{\check{\diamond}} \nc{\catpr}{\diamond_l}
\nc{\rcatpr}{\diamond_r} \nc{\lapr}{\diamond_a}
\nc{\sqcupm}{\ot}
\nc{\lepr}{\diamond_e} \nc{\vep}{\varepsilon} \nc{\labs}{\mid\!}
\nc{\rabs}{\!\mid} \nc{\hsha}{\widehat{\sha}}
\nc{\lsha}{\stackrel{\leftarrow}{\sha}}
\nc{\rsha}{\stackrel{\rightarrow}{\sha}} \nc{\lc}{\lfloor}
\nc{\rc}{\rfloor}
\nc{\tpr}{\sqcup}
\nc{\nctpr}{\vee}
\nc{\plpr}{\star}
\nc{\rbplpr}{\bar{\plpr}}
\nc{\sqmon}[1]{\langle #1\rangle}
\nc{\forest}{\calf}
\nc{\altx}{\Lambda_X} \nc{\vecT}{\vec{T}} \nc{\onetree}{\bullet}
\nc{\Ao}{\check{A}}
\nc{\seta}{\underline{\Ao}}
\nc{\deltaa}{\overline{\delta}}
\nc{\trho}{\tilde{\rho}}

\nc{\rpr}{\circ}
\nc{\dpr}{{\tiny\diamond}}
\nc{\rprpm}{{\rpr}}

\nc{\mmbox}[1]{\mbox{\ #1\ }} \nc{\ann}{\mrm{ann}}
\nc{\Aut}{\mrm{Aut}} \nc{\can}{\mrm{can}}
\nc{\twoalg}{{two-sided algebra}\xspace}
\nc{\colim}{\mrm{colim}}
\nc{\Cont}{\mrm{Cont}} \nc{\rchar}{\mrm{char}}
\nc{\cok}{\mrm{coker}} \nc{\dtf}{{R-{\rm tf}}} \nc{\dtor}{{R-{\rm
tor}}}
\renewcommand{\det}{\mrm{det}}
\nc{\depth}{{\mrm d}}
\nc{\Div}{{\mrm Div}} \nc{\End}{\mrm{End}} \nc{\Ext}{\mrm{Ext}}
\nc{\Fil}{\mrm{Fil}} \nc{\Frob}{\mrm{Frob}} \nc{\Gal}{\mrm{Gal}}
\nc{\GL}{\mrm{GL}} \nc{\Hom}{\mrm{Hom}} \nc{\hsr}{\mrm{H}}
\nc{\hpol}{\mrm{HP}} \nc{\id}{\mrm{id}} \nc{\im}{\mrm{im}}
\nc{\incl}{\mrm{incl}} \nc{\length}{\mrm{length}}
\nc{\LR}{\mrm{LR}} \nc{\mchar}{\rm char} \nc{\NC}{\mrm{NC}}
\nc{\mpart}{\mrm{part}} \nc{\pl}{\mrm{PL}}
\nc{\ql}{{\QQ_\ell}} \nc{\qp}{{\QQ_p}}
\nc{\rank}{\mrm{rank}} \nc{\rba}{\rm{RBA }} \nc{\rbas}{\rm{RBAs }}
\nc{\rbpl}{\mrm{RBPL}}
\nc{\rbw}{\rm{RBW }} \nc{\rbws}{\rm{RBWs }} \nc{\rcot}{\mrm{cot}}
\nc{\rest}{\rm{controlled}\xspace}
\nc{\rdef}{\mrm{def}} \nc{\rdiv}{{\rm div}} \nc{\rtf}{{\rm tf}}
\nc{\rtor}{{\rm tor}} \nc{\res}{\mrm{res}} \nc{\SL}{\mrm{SL}}
\nc{\Spec}{\mrm{Spec}} \nc{\tor}{\mrm{tor}} \nc{\Tr}{\mrm{Tr}}
\nc{\mtr}{\mrm{sk}}

\nc{\ab}{\mathbf{Ab}} \nc{\Alg}{\mathbf{Alg}}
\nc{\Algo}{\mathbf{Alg}^0} \nc{\Bax}{\mathbf{Bax}}
\nc{\Baxo}{\mathbf{Bax}^0} \nc{\RB}{\mathbf{RB}}
\nc{\RBo}{\mathbf{RB}^0} \nc{\BRB}{\mathbf{RB}}
\nc{\Dend}{\mathbf{DD}} \nc{\bfk}{{K}} \nc{\bfone}{{\bf 1}}
\nc{\base}[1]{{a_{#1}}} \nc{\detail}{\marginpar{\bf More detail}
    \noindent{\bf Need more detail!}
    \svp}
\nc{\Diff}{\mathbf{Diff}} \nc{\gap}{\marginpar{\bf
Incomplete}\noindent{\bf Incomplete!!}
    \svp}
\nc{\FMod}{\mathbf{FMod}} \nc{\mset}{\mathbf{MSet}}
\nc{\rb}{\mathrm{RB}} \nc{\Int}{\mathbf{Int}}
\nc{\Mon}{\mathbf{Mon}}
\nc{\remarks}{\noindent{\bf Remarks: }}
\nc{\OS}{\mathbf{OS}} 
\nc{\Rep}{\mathbf{Rep}}
\nc{\Rings}{\mathbf{Rings}} \nc{\Sets}{\mathbf{Sets}}
\nc{\DT}{\mathbf{DT}}

\nc{\BA}{{\mathbb A}} \nc{\CC}{{\mathbb C}} \nc{\DD}{{\mathbb D}}
\nc{\EE}{{\mathbb E}} \nc{\FF}{{\mathbb F}} \nc{\GG}{{\mathbb G}}
\nc{\HH}{{\mathbb H}} \nc{\LL}{{\mathbb L}} \nc{\NN}{{\mathbb N}}
\nc{\QQ}{{\mathbb Q}} \nc{\RR}{{\mathbb R}} \nc{\BS}{{\mathbb{S}}} \nc{\TT}{{\mathbb T}}
\nc{\VV}{{\mathbb V}} \nc{\ZZ}{{\mathbb Z}}


\nc{\calao}{{\mathcal A}} \nc{\cala}{{\mathcal A}}
\nc{\calc}{{\mathcal C}} \nc{\cald}{{\mathcal D}}
\nc{\cale}{{\mathcal E}} \nc{\calf}{{\mathcal F}}
\nc{\calfr}{{{\mathcal F}^{\,r}}} \nc{\calfo}{{\mathcal F}^0}
\nc{\calfro}{{\mathcal F}^{\,r,0}} \nc{\oF}{\overline{F}}
\nc{\calg}{{\mathcal G}} \nc{\calh}{{\mathcal H}}
\nc{\cali}{{\mathcal I}} \nc{\calj}{{\mathcal J}}
\nc{\call}{{\mathcal L}} \nc{\calm}{{\mathcal M}}
\nc{\caln}{{\mathcal N}} \nc{\calo}{{\mathcal O}}
\nc{\calp}{{\mathcal P}} \nc{\calq}{{\mathcal Q}} \nc{\calr}{{\mathcal R}}
\nc{\calt}{{\mathcal T}} \nc{\caltr}{{\mathcal T}^{\,r}}
\nc{\calu}{{\mathcal U}} \nc{\calv}{{\mathcal V}}
\nc{\calw}{{\mathcal W}} \nc{\calx}{{\mathcal X}}
\nc{\CA}{\mathcal{A}}

\nc{\fraka}{{\mathfrak a}} \nc{\frakB}{{\mathfrak B}}
\nc{\frakb}{{\mathfrak b}} \nc{\frakd}{{\mathfrak d}}
\nc{\oD}{\overline{D}}
\nc{\frakF}{{\mathfrak F}} \nc{\frakg}{{\mathfrak g}}
\nc{\frakm}{{\mathfrak m}} \nc{\frakM}{{\mathfrak M}}
\nc{\frakMo}{{\mathfrak M}^0} \nc{\frakp}{{\mathfrak p}}
\nc{\frakS}{{\mathfrak S}} \nc{\frakSo}{{\mathfrak S}^0}
\nc{\fraks}{{\mathfrak s}} \nc{\os}{\overline{\fraks}}
\nc{\frakT}{{\mathfrak T}}
\nc{\oT}{\overline{T}}
\nc{\frakX}{{\mathfrak X}} \nc{\frakXo}{{\mathfrak X}^0}
\nc{\frakx}{{\mathbf x}}
\nc{\frakTx}{\frakT}      
\nc{\frakTa}{\frakT^a}        
\nc{\frakTxo}{\frakTx^0}   
\nc{\caltao}{\calt^{a,0}}   
\nc{\ox}{\overline{\frakx}} \nc{\fraky}{{\mathfrak y}}
\nc{\frakz}{{\mathfrak z}} \nc{\oX}{\overline{X}}

\font\cyr=wncyr10

\nc{\al}{\alpha}
\nc{\lam}{\lambda}
\nc{\lr}{\longrightarrow}


\title[Transposed Poisson algebras]{Transposed Poisson algebras, Novikov-Poisson algebras and 3-Lie algebras}

\author{Chengming Bai}
\address{Chern Institute of Mathematics \& LPMC, Nankai University, Tianjin 300071, China}
         \email{baicm@nankai.edu.cn}

\author{Ruipu Bai}
\address{College of Mathematics and Information Science,
Hebei University, Baoding 071002, P.R. China}
\email{bairp1@yahoo.com.cn}

\author{Li Guo}
\address{Department of Mathematics and Computer Science, Rutgers University, Newark, NJ 07102, USA}
         \email{liguo@rutgers.edu}

\author{Yong Wu }
\address{Hebei College of Science and Technology, Baoding 071002, P.R. China}
\email{wuyg1022@sina.com}

\date{\today}

\begin{abstract}
We introduce a dual notion of the Poisson algebra by exchanging
the roles of the two binary operations in the Leibniz rule
defining the Poisson algebra. We show that the \dualp algebra thus
defined not only shares common properties of the Poisson algebra,
including the closure under taking tensor products and the Koszul
self-duality as an operad, but also admits a rich class of
identities. More significantly, a \dualp algebra naturally arises
from a Novikov-Poisson algebra by taking the commutator Lie
algebra of the Novikov algebra. Consequently, the classic
construction of a Poisson algebra from a commutative associative
algebra with a pair of commuting derivations has a similar
construction of a \dualp algebra when there is one derivation.
More broadly, the \dualp algebra also captures the algebraic
structures when the commutator is taken in pre-Lie Poisson
algebras and two other Poisson type algebras. Furthermore, the
transposed Poisson algebra improves two processes in~\mcite{Dz}
that produce 3-Lie algebras from Poisson algebras with a
strongness condition. When transposed Poisson algebras are used in
one process, the strongness condition is no longer needed and the
resulting 3-Lie algebra gives a transposed Poisson 3-Lie algebra.
In the other process, the resulting 3-Lie algebra is shown to again give a transposed Poisson 3-Lie
algebra.
\end{abstract}

\subjclass[2020]{
17B63,  
17A36,  
18M70,  
17D25,  
17A40,  
37J39,   
53D17  
}

\keywords{Lie algebra, Poisson algebra, \dualp algebra, pre-Lie algebra, Novikov algebra}

\maketitle


\tableofcontents

\allowdisplaybreaks

\section{Introduction}
This paper studies a variation of the well-known Poisson algebra in which the Leibniz rule of the Lie bracket on the commutative associative
multiplication is replaced by a rescaling of the Leibniz rule of
the commutative associative multiplication on the Lie bracket.

We first recall the definition of a Poisson algebra.

\begin{defi}\mlabel{de:Poi}
Let $L$ be a vector space equipped with two bilinear operations
$$
\cdot,\; [\;,\;] :L\otimes L\to L.$$
The triple $(L,\cdot,[\;,\;])$ is called a
\textbf{Poisson algebra} if $(L,\cdot)$ is a commutative associative algebra and
$(L,[\;,\;])$ is a Lie algebra that satisfy the compatibility condition
\begin{equation}\mlabel{eq:LR}
[x,y\cdot z]=[x,y]\cdot z+y\cdot [x,z],\quad \forall x,y,z\in L.
\end{equation}
\end{defi}
Eq.~(\mref{eq:LR}) is called the {\bf Leibniz
rule} since the adjoint operators of the Lie algebra are
derivations of the commutative associative algebra.

Poisson algebras arose from the study of Poisson
geometry~\mcite{Li77,Wei77} in the 1970s and has appeared in an extremely wide range of areas in
mathematics and physics, such as Poisson
manifolds~\mcite{BV1,L1,Vaisman1}, algebraic
geometry~\cite{BLLM,GK04,Pol97}, operads~\mcite{Fr,GR,MR},
quantization theory~\cite{Hue90,Kon03}, quantum
groups~\cite{CP1,Dr87}, and classical and quantum
mechanics~\cite{Arn78,Dirac64,OdA}.

The study of Poisson algebras also led to other algebraic
structures, such as noncommutative Poisson algebras~\mcite{XuP},
Jacobi algebras (also called generalized Poisson
algebras)~\mcite{AM,CK,Kir,Li77}, Gerstenhaber algebras and
Lie-Rinehart algebras~\mcite{Ger,KS,LV,Ri}. Especially interesting
is the notion of a Novikov-Poisson algebras~\cite{BCZ1,XuX}
arising from giving a tensor product structure to the important
notion of a Novikov algebra in connection with the Poisson
brackets of hydrodynamic type~\mcite{BN} and Hamiltonian operators
in the formal variational calculus~\mcite{GD}.

The Novikov-Poisson algebra~\mcite{XuX} is obtained from replacing the Lie
bracket in a Poisson algebra by the Novikov algebra product, and
replacing the Leibniz rule (\ref{eq:LR}) by certain compatibility
conditions. As a Lie-admissible algebra, taking commutators in a
Novikov algebra gives a Lie algebra. Thus it is natural to expect
that taking commutators in the Novikov product of a
Novikov-Poisson algebra gives a Poisson algebra. It came as a
surprise for us to discover that this is not the case, but rather
a ``dual" structure of the Poisson algebra in the sense that the
compatibility condition, or Leibniz rule in Eq~(\mref{eq:LR}) of a
Poisson algebra is modified by exchanging the roles of the two
binary operations (commutative associative and Lie operations).
Even more unexpected is the fact that the multiplication
operators of the commutative associative algebra are a rescaling
of derivations of the Lie algebra, rather than the derivations
themselves. This motivated us to introduce the following notion.

\begin{defi}
Let $L$ be a vector space equipped with two bilinear operations
$$
\cdot,\; [\;,\;] :L\otimes L\to L.$$
The triple $(L,\cdot,[\;,\;])$ is called a
\textbf{\ldualp algebra} or {\bf \dualp algebra} in short if $(L,\cdot)$ is a commutative associative algebra and
$(L,[\;,\;])$ is a Lie algebra that satisfy the following compatibility condition
\begin{equation}
2z\cdot [x,y]=[z\cdot x,y]+[x,z\cdot y],\;\;\forall x,y,z\in
L.\mlabel{eq:dualp}
\end{equation}
\end{defi}

Eq.~\meqref{eq:dualp} is called the {\bf transposed
Leibniz rule} because the roles played by the two binary operations in the Leibniz rule in a Poisson algebra are switched. Further, the resulting operation is rescaled by introducing a factor 2 on the left hand side.

In this paper we present three aspects of \dualp algebras that demonstrate that, despite the unusual appearance of the identity \meqref{eq:dualp}, the new structure has very good properties at the levels of algebras, operads and categories, and is the right algebraic structure to expect in connection with several important algebraic structures including Novikov-Poisson algebras, 3-Lie algebras, as well as Poisson algebras.

First in Section~\mref{sec:prop}, we consider basic properties,
examples and identities of \dualp algebras. The notion of a
Poisson algebra was motivated by the action of commuting
derivations from geometry and mechanics. The notion of a \dualp
algebra also has similar examples from classical analysis. In
particular, for a commutative associative algebra $(L,\cdot)$ with
a derivation $D$, define the Lie bracket by (see Proposition~\mref{pp:derdual})
\begin{equation}
 [x,y]=x\cdot D(y)-y\cdot D(x),\;\;\forall x,y\in L.
\end{equation}
Then $(L,\cdot,[\;,\;])$ is a \dualp algebra. Further, as one of
the important properties of Poisson algebras, there exists a natural Poisson algebra structure on the tensor product space of two
Poisson algebras which overcomes the lack of such a property for Lie algebras. The \dualp algebra also has this property
(Theorem~\mref{thm:tensor}). A remarkable property of the Poisson
algebra is the Koszul self-duality of its operad, which is shared by the \dualp algebra (Proposition~\mref{pp:dualop}). On the other
hand, while both the Poisson algebra and \dualp algebra are built from a commutative associative operation and a Lie operation, the
coupling relations of the two operations in the two algebras are orthogonal in a suitable sense (Proposition~\mref{pp:inter0}).
Furthermore, \dualp algebras possess an intriguing class of algebraic identities
(Theorem~\mref{thm:id}), some of which turn out to be critical for existing algebraic structures studied later in the paper.

The second important aspect of the \dualp algebra is its close relationship with several other algebraic structures akin to the Poisson algebra, as presented in Section~\mref{sec:novi}.
As noted above, taking the commutator in a Novikov-Poisson algebra gives rise to a \dualp algebra.
Another important Lie-admissible algebra, which is also more general than the Novikov algebra, is
the pre-Lie algebra, as the fundamental algebraic structure in many fields in
mathematics and mathematical physics, such as differential geometry, deformation theory, operads, vertex algebras, quantum field theory, symplectic structures and phase spaces, just to name a few (see \mcite{Bu} and the references therein).
From the study of Hopf algebras of rooted trees in quantum field
theory and control theory, as well as of
$F$-manifolds~\mcite{Foissy,Mansuy,Do}, the
pre-Lie commutative algebra was introduced, to consist of a
commutative associative operation and a pre-Lie operation whose
compatibility conditions are still given by the Leibniz rule, that
is, the left multiplication operators of the pre-Lie algebra are
derivations of the commutative associative algebra. In this sense,
the pre-Lie commutative algebra is very similar to the Poisson
algebra. Another related structure is the so-called differential
Novikov-Poisson algebras introduced in~\cite{BCZ1}.

To provide a uniform approach, we introduce the notion of a
pre-Lie Poisson algebra by replacing the Lie bracket in a Poisson
algebra with the pre-Lie product and replacing the Leibniz rule
(\ref{eq:LR}) in a Poisson algebra by compatibility conditions
similar to those of the Novikov-Poisson algebra. Novikov-Poisson
algebras and differential Novikov-Poisson algebras are pre-Lie
Poisson algebras, as are the pre-Lie commutative algebras under a
suitable condition. We show that taking the commutator in a
pre-Lie Poisson algebra gives rise to a \dualp algebra
(Proposition~\mref{pro:LPc}). Consequently, the same is true for a
pre-Lie commutative algebra with the above condition and for a
differential Novikov-Poisson algebra, as well as for
a Novikov-Poisson algebra.

The third aspect for the importance of the \dualp algebra is that it fits particularly nicely with various 3-Lie algebras or $n$-Lie algebras, as demonstrated in Section~\mref{sec:3lie}.
The notions of a 3-Lie algebra or more
generally an $n$-Lie algebra are the $n$-ary generalizations of a
Lie algebra~\mcite{Filippov} and have played vital roles in many
fields in mathematics and physics~\mcite{N,T}. Thus it is
important to construct 3-Lie algebras or more generally, $n$-Lie
algebras, from Lie algebras or from $n$-Lie algebras in smaller
arities. The Poisson or Poisson-like structures are instrumental
for such constructions.
In~\mcite{Dz}, for any $n$-Lie algebras
$n\geq 2$, the notion
of a Poisson $n$-Lie algebra (called $n$-Lie-Poisson there) was introduced, recovering  the
Poisson algebra when $n=2$ (also see \mcite{GM,MVV,Ro}). With an additional ``strongness" condition, a Poisson $n$-Lie algebra with a derivation can be used to construct an $(n+1)$-Lie
algebra.

It is fascinating to note that this strongness condition for $n=2$
is actually an identity holding
automatically for any \dualp algebra. In fact, the same
construction of a 3-Lie algebra can be provided by any \dualp
algebra with a derivation, {\em with no constraints}
(Theorem~\mref{thm:3Lie}). In addition, this resulted 3-Lie
algebra and the commutative associative algebra satisfy the
following ternary analog of the compatibility condition in
Eq.~\meqref{eq:dualp}:
$$
3u\cdot [x,y,z]=[x\cdot u,y,z]+[x,y\cdot u,z]+[x,y,z\cdot u],
\;\;\forall ~x,y,z,u\in L,
$$
which motivates us to introduce the notion of a \dualp 3-Lie
algebra. Such an ascending process is expected to continue for \dualp $n$-Lie algebras with any $n\geq 3$ and is formulated in
Conjecture~\mref{cj:dualp}. Incidently, by putting into such a framework, the dilation factor 2 in the transposed Leibniz
rule \meqref{eq:dualp} has a natural interpretation as
the arity 2 of the Lie bracket. By a similar
process, \dualp algebras with an involutive endomorphism also give rise to 3-Lie
algebras (Theorem~\mref{thm:const3}).

Finally, 3-Lie algebras can also be derived from strong Poisson
algebras without any additional structures~\mcite{Dz}. We show that, in the resulting 3-Lie algebras,
the compatibility condition of the 3-Lie product and the
commutative associative product of the Poisson algebra is the one for a \dualp 3-Lie algebra (Theorem~\mref{thm:Poi-dualp}). In fact, only under strict
constraints, can the compatibility condition also be the one for a Poisson 3-Lie algebra (Remark~\mref{rmk:trivial-3}).

\smallskip

\noindent
{\bf Notations. }
Throughout the paper, the base field $\bfk$ are assumed to be characteristic zero. To simplify notations, the commutative associative multiplication $\cdot$ will usually be suppressed unless emphasis is needed.

\section{Identities and properties of \dualp algebras }
\mlabel{sec:prop}

This section presents examples, identities and basic properties of \dualp algebras.

\subsection{Preliminary examples and classifications of \dualp algebras}
\mlabel{ss:ex}

We first give some preliminary examples of \dualp algebras and a
classification of complex \dualp algebras in dimension
2. More examples will be provided throughout the rest of the
paper.

A classical example of Poisson algebra from mechanics and
symplectic manifolds is the Hamilton algebra~(see \mcite{MPR})
$\mathbf{H}_n:=\bfk[x_1,\cdots,x_n,y_1,\cdots,y_n]$ with its
commutative associative product and the (Poisson) bracket $$[f,g]:=\sum_{i=1}^n \partial_{x_i}(f) \partial_{y_i}(g)-\partial_{y_i}(f)\partial_{x_i}(g), \quad \forall f, g\in \mathbf{H}_n,$$
where $\partial_{z}$ is the partial derivative with respect to the variable $z\in \{x_1,\cdots,x_n,y_1,\cdots,y_n\}$.

At a more abstract level, we have

\begin{ex}
Let $(L,\cdot)$ be a commutative associative algebra and $D_1,D_2$
be commuting derivations (that is, $D_1D_2=D_2D_1$). Then there is a Lie
algebra $(L,[\;,\;])$ defined by
\begin{equation}\mlabel{eq:2der}
[x,y]=D_1(x)\cdot D_2(y)-D_1(y)\cdot D_2(x),\;\;\forall x,y\in L.
\end{equation}
Moreover, $(L,\cdot,[\;,\;])$ is a Poisson algebra.
\mlabel{ex:Poi}
\end{ex}
See~\mcite{BLLM} for a recent application to model theory and differential-algebraic geometry.

When there is only one derivation, we have

\begin{pro}
Let $(L,\cdot)$ be a commutative associative algebra and let $D$ be a derivation. Define the multiplication
\begin{equation}
[x,y]:= x\cdot D(y)-D(x)\cdot y, \quad \forall x, y\in L.
\mlabel{eq:twoder}
\end{equation}
Then $(L,\cdot,[\;,\;])$ is a \dualp algebra.
\mlabel{pp:derdual}
\end{pro}

We note that the defining identity~\meqref{eq:dualp} does not hold without the factor 2. See Section~\mref{sec:novi} for the perspective from Novikov-Poisson algebras.

\begin{proof}
All the needed identities can be verified directly, but we will
give the proposition as a consequence
(Corollary~\mref{co:dualpSG}) of the more general result in
Theorem~\mref{thm:NPc} on Novikov-Poisson algebras. For now, we
just verify Eq.~\meqref{eq:dualp} to illustrate that the factor 2
in the equation is indispensable. For all $x,y\in L$, we
have
\begin{eqnarray*}
[z\mcdot x,y]+[x,z\mcdot y]&=&
z\mcdot x \mcdot D(y)-D(z\mcdot x)\mcdot y + x\mcdot D(z\mcdot y)-D(x)\mcdot z\mcdot y\\
&=&
 z \mcdot x\mcdot D(y)-z\mcdot D(x)\mcdot y-D(z)\mcdot x\mcdot y+ x\mcdot z\mcdot D(y)+ x\mcdot D(z)\mcdot y -D(x)\mcdot z\mcdot y\\
&=& 2 z\mcdot (x\mcdot D(y)-D(x)\mcdot y)\\
&=& 2 z\mcdot [x,y]. \hspace{9.5cm} \qedhere
\end{eqnarray*}
\end{proof}

As a concrete example, for the usual derivations, we have

\begin{ex} The commutative associative algebra $L=\bfk[x_1,\cdots, x_n]$ together with the bracket
$$[g,h]:=\sum_{i=1}^n (g \partial_{x_i}(h)-h\partial_{x_i}(g)), \quad \forall g, h\in L,$$
is a \dualp algebra. More generally, for
$D=\sum_{i=1}^nf_i\partial_{x_i}$ with $f_i\in L, 1\leq i\leq n$, the bracket $$[g,h]:=\sum_{i=1}^n(f_i(g\partial_{x_i}(h)-h\partial_{x_i}(g)),\;\;\forall
g,h\in L,$$
also gives a \dualp algebra.
\end{ex}

As special cases of \dualp algebras, by inspection, we have

\begin{pro}\label{pro:trivial}
For a linear space $L$ with a commutative associative
multiplication $\cdot$ and a Lie bracket $[\;,\;]$.
\begin{enumerate}
\item
If either $\cdot$ or $[\;,\;]$
is zero, then $(L,\cdot,[\;,\;])$ is a \dualp algebra, as well as a Poisson algebra.
\mlabel{it:trivial1}
\item
More generally, if
$$ x\cdot [y,z]=[x\cdot y,z]=0, \quad \forall x,y,z\in L,$$
then $(L,\cdot,[\;,\;])$ is a \dualp algebra, as well as a Poison algebra.
\mlabel{it:trivial2}
\end{enumerate}
\end{pro}

As we will see in Proposition~\mref{pp:inter0}, the converse of Proposition~\mref{pro:trivial}.\meqref{it:trivial2} holds.
\smallskip

We next classify the \dualp algebra structures on two-dimensional complex Lie algebras.

Let $L$ be a complex vector space with a basis $\{e_1,e_2\}$. If
$(L,[\;,\;])$ is an abelian Lie algebra, then by
Proposition~\ref{pro:trivial}.\meqref{it:trivial1}, for any
2-dimensional commutative associative algebra $(L,\cdot)$,
$(L,\cdot, [\;,\;])$ is a \dualp algebra. Hence any 2-dimensional complex \dualp algebra in which the Lie algebra
is abelian is isomorphic to one of the following \dualp algebras
(we only give non-zero product):
\begin{enumerate}
 \item $(L,\cdot): e_i\mcdot e_j=0, i,j=1,2;\;\;(L,[\;,\;]): [e_1,e_2]=0$;
\mlabel{it:2dim11} \item $(L,\cdot): e_1\mcdot e_1=e_1,e_2\mcdot
e_2=e_2,\;\;(L,[\;,\;]): [e_1,e_2]=0$; \mlabel{it:2dim22} \item
$(L,\cdot): e_1\mcdot e_1=e_1,e_1\mcdot e_2=e_2\mcdot
e_1=e_2,\;\;(L,[\;,\;]): [e_1,e_2]=0$; \mlabel{it:2dim32} \item
$(L,\cdot): e_1\mcdot e_1=e_1,\;\;(L,[\;,\;]): [e_1,e_2]=0$;
\mlabel{it:2dim42} \item $(L,\cdot): e_1\mcdot e_1=
e_2,\;\;(L,[\;,\;]): [e_1,e_2]=0$. \mlabel{it:2dim43}
\end{enumerate}
These \dualp algebras are mutually non-isomorphic.

On the other
hand, it is well-known that any 2-dimensional
non-abelian complex Lie algebra is isomorphic to the one whose product is given by
$$[e_1,e_2]=e_2.$$
By a straightforward computation, we find that a commutative (not
necessarily associative) algebra $(L,\cdot)$ satisfying
Eq.~(\mref{eq:dualp}) if and only if its product is given by
$$e_1\mcdot e_1=ae_1+be_2,\;e_1\mcdot e_2=e_2\mcdot e_1=ce_1+ae_2,\;e_2\mcdot e_2=ce_2,$$
for some $a, b, c\in \CC.$ Moreover, $(L,\cdot)$ is associative if
and only if $ac=0.$ Therefore, any 2-dimensional complex \dualp
algebra in which the Lie algebra is non-abelian is isomorphic to
one of the following \dualp algebras (we only give the
non-zero product):

\begin{enumerate}
 \item $(L,\cdot): e_i\mcdot e_j=0, i,j=1,2;\;\;(L,[\;,\;]): [e_1,e_2]=e_2$;
\mlabel{it:2dim1} \item $(L,\cdot): e_1\mcdot
e_1=e_2,\;\;(L,[\;,\;]): [e_1,e_2]=e_2$; \mlabel{it:2dim2} \item
$(L,\cdot): e_1\mcdot e_2=e_2\mcdot e_1=e_1,e_2\mcdot
e_2=e_2,\;\;(L,[\;,\;]): [e_1,e_2]=e_2$; \mlabel{it:2dim3} \item
$(L,\cdot): e_1\mcdot e_1=\lambda e_1, e_1\mcdot e_2=e_2\mcdot
e_1=\lambda e_2, \lambda\ne 0,\;\;(L,[\;,\;]): [e_1,e_2]=e_2$.
\mlabel{it:2dim4}
\end{enumerate}
\noindent All parameters appearing above are in $\mathbb C$. Furthermore, these \dualp algebras are mutually non-isomorphic.

\subsection{Identities in \dualp algebras}
\mlabel{ss:id}

There is a rich family of identities for \dualp
algebras which will be used later in the paper, but are also
interesting in their own right.

\begin{thm} Let $(L,\cdot, [\;,\;])$ be a \dualp
algebra. Then the following identities hold.
\begin{eqnarray}
&&x\mcdot[y,z]+y\mcdot[z,x]+z\mcdot[x,y]=0,\mlabel{eq:gi1}\\
&&[h\mcdot[x,y],z]+[h\mcdot[y,z],x]+[h\mcdot [z,x],y]=0,\mlabel{eq:gi2}\\
&&[h\mcdot x,[y,z]]+[h\mcdot y,[z,x]]+[h\mcdot z,[x,y]]=0,\mlabel{eq:gi3}\\
&&[h,x]\mcdot [y,z]+[h,y]\mcdot [z,x]+[h,z]\mcdot [x,y]=0,\mlabel{eq:gi4}\\
&&[x\mcdot u,y\mcdot v]+[x\mcdot v,y\mcdot u]=2u\mcdot v\mcdot
[x,y],\mlabel{eq:gi5}\\
&& x[u,yv]+v[xy,u]+yu[v,x]=0, \mlabel{eq:gi6}
\end{eqnarray}
for all $x,y,z,h,u,v\in L$.
\mlabel{thm:id}
\end{thm}

\begin{proof} Let $x,y,z,h,u,v\in L$. By Eq.~(\mref{eq:dualp}), we have
$$[x\mcdot z,y]+[x,y\mcdot z]=2z\mcdot [x,y],$$
$$[y\mcdot x,z]+[y,z\mcdot x]=2x\mcdot [y,z],$$
$$[z\mcdot y,x]+[z,x\mcdot y]=2y\mcdot [z,x].$$
Taking the sum of the three identities above gives Eq.~(\mref{eq:gi1}).

The proof of Eq.~\meqref{eq:gi2} takes more efforts. First by Eq.~(\mref{eq:dualp}), we have
$$[[x,y]\mcdot h,z]+[[x,y],z\mcdot h]=2h\mcdot [[x,y],z],$$
$$[[y,z]\mcdot h,x]+[[y,z],x\mcdot h]=2h\mcdot [[y,z],x],$$
$$[[z,x]\mcdot h,y]+[[z,x],y\mcdot h]=2h\mcdot [[z,x],y].$$
Taking the sum of the three identities above  and applying the Jacobi identity, we obtain
\begin{equation}
[[x,y]\mcdot h,z]+[[y,z]\mcdot h,x]+[[z,x]\mcdot h,y])+([[x,y],z\mcdot
h]+[[y,z],x\mcdot h]+[[z,x],y\mcdot h])=0.\mlabel{eq:gin1}
\end{equation}
Next applying the Jacobi identity to $[[x,y],z\mcdot h]$ gives
$$[[x,y],z\mcdot h]+[[y,z\mcdot h],x]+[[z\mcdot h,x],y]=0$$
and applying Eq.~(\mref{eq:dualp}) to $[[y,z\mcdot h],x]$ gives
$$[[y,z\mcdot h],x]=2[h\mcdot [y,z],x]-[[y\mcdot h,z],x].$$
Thus we obtain
$$[[x,y],z\mcdot h]+2[h\mcdot [y,z],x]-[[y\mcdot h,z],x]+[[z\mcdot h,x],y]=0.$$
Similarly, we have
$$[[y,z],x\mcdot h]+2[h\mcdot [z,x],y]-[[z\mcdot h,x],y]+[[x\mcdot h,y],z]=0,$$
$$[[z,x],y\mcdot h]+2[h\mcdot [x,y],z]-[[x\mcdot h,y],z]+[[y\mcdot h,z],x]=0.$$
Taking the sum of the three identities above yields
\begin{equation}[[x,y],z\mcdot h]+[[y,z],x\mcdot h]+[[z,x],y\mcdot h])+2([h\mcdot [x,y],z]+[h\mcdot [y,z],x]+[h\mcdot [z,x],y])=0.
\mlabel{eq:gin2}\end{equation}
Then Eq.~(\mref{eq:gi2}) follows from
taking the difference between Eq.~(\mref{eq:gin2}) and
Eq.~(\mref{eq:gin1}).

Furthermore, Eq.~(\mref{eq:gi3}) follows by
substituting Eq.~(\mref{eq:gi2}) into Eq.~(\mref{eq:gin1}).

By Eq.~(\mref{eq:dualp}), we have
$$[h\mcdot [x,y],z]+[h,z\mcdot [x,y]]=2[x,y]\mcdot [h,z],$$
$$[h\mcdot [y,z],x]+[h,x\mcdot [y,z]]=2[y,z]\mcdot [h,x],$$
$$[h\mcdot [z,x],y]+[h,y\mcdot [z,x]]=2[z,x]\mcdot [h,y].$$
Taking the sum of the above three identities and applying
Eq.~(\mref{eq:gi2}), we obtain Eq.~(\mref{eq:gi4}).

By Eq.~(\mref{eq:dualp}), we get
$$[x\mcdot u,y\mcdot v]+[x,y\mcdot v\mcdot u]=2u\mcdot [x,y\mcdot
v],$$
$$[x\mcdot v,y\mcdot u]+[x\mcdot v\mcdot u,y]=2u\mcdot [x\mcdot v,y].$$
Taking the sum of the above two identities gives
$$([x\mcdot u,y\mcdot v]+[x\mcdot v,y\mcdot u])+([x,y\mcdot v\mcdot u]+[x\mcdot v\mcdot u,y])=2u\mcdot ([x,y\mcdot v]+[x\mcdot v,y]).$$
Since
$$[x,y\mcdot v\mcdot u]+[x\mcdot v\mcdot u,y]=2v\mcdot u\mcdot [x,y],$$
$$2u\mcdot ([x,y\mcdot v]+[x\mcdot v,y])=4u\mcdot v\mcdot [x,y],$$
Eq.~(\mref{eq:gi5}) follows.

Finally by Eqs.~\eqref{eq:dualp} and \eqref{eq:gi5}, we have Eq.~\eqref{eq:gi6}:
$$2x[u,yv]+2v[xy,u]+2yu[v,x]= [xu,yv]+[u,xyv]+[xyv,u]+[xy,uv]+2uv[x,y]
=0. \qedhere
$$
\end{proof}

\subsection{Related structures and tensor products}

The following result shows that \dualp algebra structures on a space is closed under the commutative associative product.
\begin{pro}
Let $(L,\cdot,[\;,\;])$ be a \dualp algebra. For any $h\in L$, define a new bilinear operation $[\;,\;]_h$ on $L$ by
\begin{equation} [x,y]_h=h\mcdot [x,y], \;\;\forall x,y\in L.\end{equation}
Then $(L,\cdot,[\;,\;]_h)$ is a \dualp algebra.
\end{pro}

\begin{proof} It is obvious that the operation $[\;,\;]_h$ is skew-symmetric.
By Eq.~(\mref{eq:gi2}),  we have
\begin{eqnarray*}
 [[x,y]_h,z]_h+[[y,z]_h,x]_h+[[z,x]_h,y]_h
=h\mcdot ([h\mcdot [x,y],z]+[h\mcdot [y,z],x]+[h\mcdot
[z,x],y])=0,\end{eqnarray*} for all $x,y,z\in L$. Then
$(L,[\;,\;]_h)$ is a Lie algebra.  By Eq.~(\mref{eq:dualp}), we
have
$$[x\mcdot z,y]_h+[x,y\mcdot z]_h=h\mcdot [x\mcdot z,y]+h\mcdot [x,y\mcdot z]=2h\mcdot z\mcdot [x,y]=2z\mcdot [x,y]_h,$$
for all $x,y,z\in L$. Hence  $(L, \mcdot, [\;,\;]_h)$ is a \dualp algebra.
\end{proof}

An important property of Poisson algebras is that they are
closed under taking tensor products (see~\mcite{XuX}.)

\begin{pro}
Let  $(L_1,\cdot_1,[\;,\;]_1)$ and $(L_2,\cdot_2,[\;,\;]_2)$ be two
Poisson algebras.  Define two operations $\cdot$ and $[\;,\;]$ on $L_1\otimes L_2$ by
\begin{equation}
(x_1\otimes x_2)\cdot (y_1\otimes y_2)=x_1\cdot_1 y_1\otimes x_2\cdot_2
y_2,\mlabel{eq:tensor1}\end{equation}
\begin{equation}
[x_1\otimes x_2,y_1\otimes y_2]=[x_1,y_1]_1\otimes x_2\cdot_2
y_2+x_1\cdot_1 y_1\otimes [x_2,y_2]_2,\mlabel{eq:tensor2}\end{equation}
for all $x_1,y_1\in L_1, x_2,y_2\in L_2$.  Then $(L_1\otimes L_2,\cdot, [\;,\;])$ is a Poisson algebra.
\end{pro}

This property is also shared by \dualp algebras.
\begin{thm}\label{thm:tensor}
Let  $(L_1,\cdot_1,[\;,\;]_1)$ and $(L_2,\cdot_2,[\;,\;]_2)$ be
two \dualp algebras. Define two operations $\cdot$ and $[\;,\;]$ on $L_1\otimes L_2$ by
Eqs.~\meqref{eq:tensor1} and \meqref{eq:tensor2} respectively.
Then $(L_1\otimes L_2,\cdot, [\;,\;])$ is a \dualp algebra.
\end{thm}

\begin{proof}
For brevity, the subscripts 1 and 2 in the operations $\cdot$ and $[\;,\;]$ will suppressed, since their meaning should be clear from the context.

It is obvious that $(L_1\otimes L_2,\cdot)$ is a commutative
associative algebra. Let $x_1,y_1,z_1\in L_1, x_2,y_2,z_2\in L_2$.
Then we have
\begin{eqnarray*}
&&[[x_1\otimes x_2,y_1\otimes y_2],z_1\otimes z_2] + [[y_1\otimes
y_2,z_1\otimes z_2],x_1\otimes x_2]+[[z_1\otimes z_2,x_1\otimes
x_2],y_1\otimes y_2]\\
&&=(A1)+(A2)+(A3)+(A4),
\end{eqnarray*}
where
\begin{eqnarray*}
(A1)&=&[[x_1,y_1],z_1]\otimes x_2\mcdot y_2\mcdot z_2
+[[y_1,z_1],x_1]\otimes y_2\mcdot z_2\mcdot x_2+
[[z_1,x_1],y_1]\otimes z_2\mcdot x_2\mcdot y_2,\\
(A2)&=&x_1\mcdot y_1\mcdot z_1\otimes [[x_2,y_2],z_2] +y_1\mcdot
z_1\mcdot x_1\otimes [[y_2,z_2],x_2] +z_1\mcdot x_1\mcdot y_1\otimes
[[z_2,x_2],y_2],\\
(A3)&=& [x_1,y_1]\mcdot z_1\otimes [x_2\mcdot y_2, z_2] +
[y_1,z_1]\mcdot x_1\otimes [y_2\mcdot z_2, x_2]+ [z_1,x_1]\mcdot
y_1\otimes [z_2\mcdot x_2, y_2],\\
(A4)&=& [x_1\mcdot y_1,z_1]\otimes [x_2,y_2]\mcdot z_2+[y_1\mcdot
z_1,x_1]\otimes [y_2,z_2]\mcdot x_2 +[z_1\mcdot x_1,y_1]\otimes
[z_2,x_2]\mcdot y_2.
\end{eqnarray*}
Since $(L_1,\cdot)$, $(L_2,\cdot)$ are commutative associative
algebras and $(L_1,[\;,\;])$, $(L_2,[\;,\;])$ are Lie algebras,
$(A1)$ and $(A2)$ are zero. By Eq.~(\mref{eq:gi1}), we have \begin{eqnarray*}
&& ([y_1,z_1]\mcdot x_1)\otimes [y_2\mcdot z_2, x_2] =([z_1,x_1]\mcdot
y_1+[x_1,y_1]\mcdot z_1)\otimes [x_2,y_2\mcdot z_2],\\
&&[y_1\mcdot z_1,x_1]\otimes [y_2,z_2]\mcdot x_2=[x_1,y_1\mcdot
z_1]\otimes ([z_2,x_2]\mcdot y_2+[x_2,y_2]\mcdot z_2).
\end{eqnarray*}
Substituting the above equations into (A3) and (A4) respectively,
we have
\begin{eqnarray*}
(A3)= ([x_1,y_1]\mcdot z_1)\otimes ([x_2\mcdot y_2, z_2]+[x_2,y_2\mcdot
z_2])+([z_1,x_1]\mcdot y_1)\otimes ([z_2\mcdot x_2, y_2]+[x_2,y_2\mcdot z_2]),
\end{eqnarray*}
\begin{eqnarray*} (A4)=([x_1\mcdot
y_1,z_1]+[x_1,y_1\mcdot z_1])\otimes ([x_2,y_2]\mcdot z_2)+([z_1\mcdot
x_1,y_1]+[x_1,y_1\mcdot z_1])\otimes ([z_2,x_2]\mcdot y_2).
\end{eqnarray*}
By Eq.~(\mref{eq:dualp}), we have
$$(A3)=2([x_1,y_1]\mcdot z_1)\otimes (y_2\mcdot [x_2,z_2])+ 2([z_1,x_1]\mcdot
y_1)\otimes (z_2\mcdot [x_2,y_2]),$$
$$(A4)=2 (y_1\mcdot [x_1,z_1])\otimes ([x_2,y_2]\mcdot z_2) +2 (z_1\mcdot [x_1,y_1])\otimes ([z_2,x_2]\mcdot y_2).$$
Hence $(A3)+(A4)=0$ and
therefore $(L_1\otimes L_2,[\;,\;])$ is a Lie algebra.

Furthermore, we have
\begin{eqnarray*}
&&[(x_1\otimes x_2) \mcdot (z_1\otimes z_2), y_1\otimes
y_2]+[x_1\otimes x_2, (y_1\otimes y_2)\mcdot (z_1\otimes z_2)]\\
&&=[x_1\mcdot z_1,y_1]\otimes (x_2\mcdot z_2\mcdot y_2)+(x_1\mcdot
z_1\mcdot y_1)\otimes [x_2\mcdot z_2,y_2]+[x_1,y_1\mcdot z_1]\otimes
(x_2\mcdot y_2\mcdot z_2)\\
&&+(x_1\mcdot y_1\mcdot z_1)\otimes [x_2, y_2\mcdot
z_2]\\
&&=([x_1\mcdot z_1,y_1]+[x_1,y_1\mcdot z_1])\otimes (x_2\mcdot y_2\mcdot z_2)+(x_1\mcdot y_1\mcdot z_1)\otimes ([x_2\mcdot
z_2,y_2]+[x_2, y_2\mcdot z_2])\\
&&=2(z_1\mcdot[x_1,y_1])\otimes (x_2\mcdot y_2\mcdot z_2) +2(x_1\mcdot
y_1\mcdot z_1)\otimes (y_2\mcdot[x_2,z_2])\\
&&=2(z_1\otimes z_2) \mcdot [x_1\otimes x_2,y_1\otimes y_2].
\end{eqnarray*}
Therefore the conclusion holds.
\end{proof}

We end this section by displaying a close relationship between \dualp algebras and Hom-Lie algebras.

Recall that a {\bf Hom-Lie algebra}~\mcite{HLS} is a triple $({\frak g},[\;,\;],\varphi)$ consisting of a
  linear space $\frak g$, a skew-symmetric bilinear operation $[\;,\;]:\wedge^2\frak g\rightarrow
  \frak g$ and a linear map $\varphi_:\frak g\rightarrow\frak g$ satisfying
  \begin{equation}
    [\varphi(x),[y,z]]+[\varphi(y),[z,x]]+[\varphi(z),[x,y]]=0,\;\;\forall x,y,z\in \frak g.
  \end{equation}
If in addition, $\varphi$ is an algebra homomorphism, then the
Hom-Lie algebra $({\frak g},[\;,\;],\varphi)$ is called {\bf
multiplicative}.

\begin{pro}Let $(L,\cdot,[\;,\;])$ be a \dualp algebra. For any
$h\in L$, define a linear map $\varphi_{h}: L\rightarrow
L$ by
\begin{equation}\mlabel{eq:varphi}
\varphi_{h}(x)=h\mcdot x, \;\;\forall x\in L.
\end{equation} Then
$(L,[\;,\;],\varphi_{h})$ is a Hom-Lie algebra. Moreover, $\varphi_h$ satisfies
\begin{equation}
\varphi_{h}^{2}([x,y])=[\varphi_{h}(x),\varphi_{h}(y)],\;\;\forall
x,y\in L.\mlabel{eq:varphi2}
\end{equation}
Hence if $\varphi_h^2=\varphi_h$, that is, $h\mcdot h\mcdot x=h\mcdot
x$ for all $x\in L$, then $(L,[\;,\;],\varphi_{h})$ is a
multiplicative Hom-Lie algebra.
\end{pro}

\begin{proof}
Let $x,y,z\in L$. Then by Eq.~(\mref{eq:gi3}), we have
$$[\varphi_{h}(x),[y,z]]+[\varphi_{h}(y),[z,x]]+[\varphi_{h}(z),[x,y]]=0.
$$
Hence $(L,[\;,\;],\varphi_{h})$ is a Hom-Lie algebra. By Eq.~(\mref{eq:gi5}), we have
$$[x\mcdot h,y\mcdot h]+[x\mcdot h, y\mcdot h]=2h\mcdot h\mcdot [x,y].$$
Hence Eq.~(\mref{eq:varphi2}) holds. Thus the conclusion follows.
\end{proof}

\subsection{Operadic and categorial properties of \dualp algebras}
\subsubsection{Operadic dual of \dualp algebra}

It is well known that the operad of Poisson algebras is
self dual~\mcite{LV}. It is interesting to see that the same is true for the operad of \dualp algebras.

\begin{pro}
The operad of \dualp algebras, denoted by \dualpop, is self dual.
\mlabel{pp:dualop}
\end{pro}

\begin{proof}
The operad $\dualpop$ is binary quadratic, with its two-dimensional operation space $V=V(2)$ spanned by the binary operations
$\mu$ for the commutative associative
operation and $\nu$ for the Lie bracket. So $\mu$ satisfies the
commutative associative law and $\nu$ is skew-symmetric and satisfies the Jacobi
identity. Further, the operations $\mu$ and $\nu$ satisfy the compatibility relation (or distributive law)
\begin{equation}
2\mu(x,\nu(y,z))=\nu(\mu(x,y),z)+\nu(y,\mu(x,z)),
\mlabel{eq:dualpop1}
\end{equation}
which is rewritten in the standard form as in~\cite[\S 7.6.2]{LV}
\begin{equation}
2\mu(\nu(y,z),x)-\nu(\mu(x,y),z)-\nu(\mu(z,x),y),
\mlabel{eq:dualpop2}
\end{equation}
that is,
$$
2\mu\cirii \nu - \nu\ciri \mu + \nu\ciriii \mu.
$$
Taking into account of the action of the symmetric group $S_3$, the relation space of the binary quadratic operad \dualpop   is of dimensional three, spanned by the basis elements
$$ 2\mu\cirii \nu - \nu\ciri \mu + \nu\ciriii \mu, \quad 2\mu\ciriii \nu - \nu\cirii \mu + \nu\ciri \mu,
\quad
2\mu\ciri \nu - \nu\ciriii \mu + \nu\cirii \mu.$$

For the dual space $V=\bfk\{\check{\mu}, \check{\nu}\}$, by the duality between the commutative associative operad and the Lie operad, $\check{\mu}$ is a Lie bracket and $\check{\nu}$ is commutative associative. So the corresponding compatibility relations have its basis
$$ 2\check{\nu}\cirii \check{\mu} - \check{\mu}\ciri \check{\nu} + \check{\mu}\ciriii \check{\nu}, \quad
2\check{\nu}\ciriii \check{\mu} - \check{\mu}\cirii \check{\nu} + \check{\mu}\ciri \check{\nu},
\quad
2\check{\nu}\ciri \check{\mu} - \check{\mu}\ciriii \check{\nu} + \check{\mu}\cirii \check{\nu}.$$
By direct inspection, we see that the two sets of relations annihilate each other under the operad pairing given by
$$ \langle \check{\mu}, \mu \rangle =\langle \check{\nu}, \nu \rangle =1, \langle \check{\mu}, \nu\rangle =\langle\check{\nu},\mu\rangle =0 .$$
Since both sets span a subspace of dimension three out of a space of dimension six, they are the exact annihilators of each other. Since the commutative associative operad is already dual to the Lie operad, this proves that the Koszul dual of \dualpop is itself.
\end{proof}

\subsubsection{Independence of Poisson algebra and \dualp algebra}
\mlabel{ss:indep} We show that the
compatibility relations of the Poisson algebra and those of the \dualp algebra are independent in the following sense.

\begin{pro} Let $(L,\cdot)$ be a commutative associative algebra and $(L,[\;,\;])$ be a Lie algebra. Then $(L,\cdot,[\;,\;])$ is both a
Poisson algebra and a \dualp algebra if and only
if
\begin{equation}
x\mcdot [y,z]=[x\mcdot y,z]=0,\;\;\forall x,y,z\in L.
\mlabel{eq:inter0}
\end{equation}
\mlabel{pp:inter0}
\end{pro}

\begin{proof}
The ``if" part has been given in
Proposition~\mref{pro:trivial}.\meqref{it:trivial2}. For the ``only if" part, let $x,y,z\in L$. By Eq.~(\mref{eq:LR}), we have
$$[z\mcdot x,y]=-x\mcdot [y,z]-z\mcdot [y,x],\;\;[x,z\mcdot y]=y\mcdot [x, z]
+z\mcdot [x,y].$$ Then by Eq.~(\mref{eq:dualp}), we have
$$0=[z\mcdot x,y]+[x,z\mcdot y]-2z\mcdot [x,y]=[x,z]\mcdot y-[y,z]\mcdot x.$$
By Eq.~(\mref{eq:gi1}), we have
$$[x,y]\mcdot z= [x,z]\mcdot y-[y,z]\mcdot x=0.$$
By Eq.~(\mref{eq:LR}) again, we have $[x\mcdot y,z]=0$.
\end{proof}

\begin{rmk}\mlabel{rmk:inter}
The above conclusion means that, with the relations of the
commutative associative product and the Lie bracket, the
intersection of the compatibility conditions of the Poisson
algebra and of the \dualp
algebra is trivial. That is, the structures of
\dualp algebras and Poisson algebras are completely inconsistent.
To put it in another context, the operad of \dualp algebras
together with Poisson algebras is the direct sum~\mcite{LV} of the
operad of the commutative associative algebra and that of
the Lie algebra.
\end{rmk}

\section{From Novikov-Poisson algebras and pre-Lie Poisson algebras to \dualp algebras}
\mlabel{sec:novi}
We now show that, for several Poisson related
algebraic structures, taking the commutator gives
rise to \dualp algebras.

\begin{defi} \mcite{XuX} A {\bf  Novikov-Poisson algebra} is
a triple $(L,\cdot, \circ)$, where $L$ is a vector space and
$\mcdot, \circ$ are two bilinear operations on $L$ satisfying the following conditions:
\begin{enumerate}
\item[(1)] $(L,\cdot)$ is a commutative associative algebra.
\item[(2)] $(L,\circ)$ is a Novikov algebra, that is, for all
$x,y,z\in L$,
\begin{equation}(x\circ y)\circ z-(y\circ x)\circ z=x\circ(y\circ z)-y\circ(x\circ z),\mlabel{eq:PLie}\end{equation}
\begin{equation} (x\circ y)\circ z=(x\circ z)\circ y.\mlabel{eq:Novikov}\end{equation}
\item[(3)] The following  equations hold for all $x,y,z\in L$.
\begin{equation}
(x\mcdot y)\circ z=x\mcdot(y\circ z),\mlabel{eq:NP1} \end{equation}
\begin{equation}(x\circ y)\mcdot z-(y\circ x)\mcdot z=x\circ(y\mcdot z)-y\circ(x\mcdot z). \mlabel{eq:NP2}\end{equation}
\end{enumerate}
\end{defi}

Taking the commutator in a Novikov-Poisson algebra, we obtain

\begin{thm}\mlabel{thm:NPc}
Let $(L,\cdot,\circ)$ be a Novikov-Poisson algebra. Define
\begin{equation}[x,y]=x\circ y-y\circ x,\;\; \forall x,y\in L.\mlabel{eq:com}\end{equation}
Then
$(L,\cdot,[\;,\;])$ is a \dualp algebra.
\end{thm}

\begin{proof}
It is known that $(L,[\;,\;])$ is a Lie algebra (cf. \mcite{Bu}, in
fact, it is due to Eq.~(\mref{eq:PLie}) only). By Eqs.~(\mref{eq:NP1})
and (\mref{eq:NP2}), for all $x,y,z\in L$, we have
\begin{eqnarray*}
&&[x\mcdot z,y]+[x,y\mcdot
z]-2z\mcdot[x,y]\\
&=&(x\mcdot z)\circ y-y\circ (x\mcdot z) +x\circ
(y\mcdot z)-(y\mcdot z)\circ x-2z\mcdot(x\circ y-y\circ x)\\
&=&((x\mcdot z)\circ y-z\mcdot(x\circ
y)+z\mcdot(y\circ x)-(y\mcdot z)\circ x)\\
&&-((x\circ y)\mcdot z-(y\circ x)\mcdot
z-x\circ (y\mcdot z)+y\circ(x\mcdot z))\\
&=&0.
\end{eqnarray*}
Hence the conclusion holds.
\end{proof}

We quote the following classical results before giving applications.

\begin{lem}\mlabel{lem:Xu}
Let $(L,\cdot)$ be a commutative associative algebra and $D$ be a derivation. Define a bilinear operation $\circ$ on $L$ by
\begin{equation}
x\circ y=x\mcdot D(y),\;\;\forall x,y\in L. \mlabel{eq:SG} \end{equation}
Then
\begin{enumerate}
\item {\rm (S.~Gelfand, \mcite{GD})}
$(L,\circ)$ is a Novikov algebra. \item {\rm \mcite{XuX}}
$(L,\cdot,\circ)$ is a Novikov-Poisson algebra.
\end{enumerate}
\end{lem}

Combining Theorem~\mref{thm:NPc} and Lemma~\mref{lem:Xu} leads to the following consequence, which can be regarded
as the basic examples of \dualp algebras in analog to the construction of Poisson algebras from commutative associative algebras with commuting derivations.

\begin{cor} \mlabel{co:dualpSG}
{\rm (=Proposition~\mref{pp:derdual})} Let $(L,\cdot)$ be a
commutative associative algebra and $D$ be a derivation. Then
$(L,\cdot,[\;,\;])$ is a  \dualp algebra, where
\begin{equation}
 [x,y]=x\mcdot D(y)-y\mcdot D(x),\;\;\forall x,y\in L.
\end{equation}
\end{cor}

For examples in low dimensions, we have
\begin{ex} The complex commutative associative algebras in dimensions in 2 and 3 and their
derivations were given in~\mcite{BM}. Hence the induced \dualp
algebras by Corollary~\mref{co:dualpSG} can be obtained explicitly.
We give the 2-dimensional cases. Let $L$ be a 2-dimensional vector
space with a basis $\{e_1,e_2\}$. In the following, we give the
non-zero multiplication of the commutative associative algebra
$(L,\cdot)$, the set ${\rm Der_a}(L)$ of all derivations of
$(L,\cdot)$ and the induced Lie algebra $(L,[\;,\;])$ such that
$(L,\cdot,[\;,\;])$ is a  \dualp algebra.
\begin{enumerate}
\item[(1)] $(L,\cdot): e_i\mcdot e_j=0,i,j=1,2;\;\; {\rm
Der_a}(L)={\rm End}(L);\;\; (L,[\;,\;]): [e_1,e_2]=0$. \item[(2)]
$(L,\cdot): e_1\mcdot e_1=e_1,e_2\mcdot e_2=e_2;\;\; {\rm
Der_a}(L)=0;\;\; (L,[\;,\;]): [e_1,e_2]=0$. \item[(3)] $(L,\cdot):
e_1\mcdot e_1=e_1,;\;\; {\rm Der_a}(L)=\left(\begin{matrix} 0& 0\cr
0& a\cr\end{matrix}\right);\;\; (L,[\;,\;]): [e_1,e_2]=0$.
\item[(4)] $(L,\cdot): e_1\mcdot e_1=e_1,e_2\mcdot e_1=e_1\mcdot
e_2=e_2;\;\; {\rm Der_a}(L)=\left(\begin{matrix} 0& 0\cr 0&
a\cr\end{matrix}\right);\;\; (L,[\;,\;]): [e_1,e_2]=ae_2$.
\item[(5)] $(L,\cdot): e_1\mcdot e_1=e_2,;\;\; {\rm
Der_a}(L)=\left(\begin{matrix} a& 0\cr b&
2a\cr\end{matrix}\right);\;\; (L,[\;,\;]): [e_1,e_2]=be_2$.
\end{enumerate}
All the parameters are in the complex number field $\mathbb C$.
\end{ex}

The following combination of the pre-Lie algebra and commutative
associative algebra arose from the study of Hopf algebras of
rooted trees in quantum field theory and control theory, as well
as of $F$-manifolds~\mcite{Foissy,Mansuy,Do}.

\begin{defi} \mcite{Mansuy}
A {\bf pre-Lie commutative algebra (or PreLie-Com algebra)} is a triple $(L,\cdot,\circ)$, where $(L,\cdot)$ is a commutative
associative algebra and $(L,\circ)$ is a pre-Lie algebra
satisfying
\begin{equation}\label{eq:PLiecomm}
  x\circ(y\mcdot z)-(x\circ y)\mcdot z-y\mcdot (x\circ z)=0,\quad\forall~x,y,z\in L.
\end{equation}
\end{defi}

Note that Eq.~(\ref{eq:PLiecomm}) means
that the left multiplication operators of the pre-Lie algebra
$(L,\circ)$ are derivations of the commutative associative algebra
$(L,\cdot)$.

\begin{ex}
Let $(L,\cdot)$ be a commutative associative algebra and $D$ be a
derivation. It is straightforward to show that $(L,\cdot,\circ)$
is a PreLie-Com algebra, where $(L,\circ)$ is defined by
Eq.~(\mref{eq:SG}): $x\circ y:=x\mcdot D(y)$ for all
$x,y\in L$.
\end{ex}

There is also a notion of a {\bf differential Novikov-Poisson algebra}
introduced in \cite[Definition 4.1]{BCZ1}, which is a triple $(L,\cdot,\circ)$, where $(L,\cdot)$ is a commutative associative algebra and $(L,\circ)$ is a Novikov algebra satisfying Eqs.~(\ref{eq:NP1}), (\ref{eq:NP2}) and
(\ref{eq:PLiecomm}).

We now introduce a general context to include Novikov-Poisson
algebras, pre-Lie commutative algebras under certain conditions and differential Novikov-Poisson algebras.

\begin{defi} A {\bf  pre-Lie Poisson algebra} is
a triple $(L,\cdot, \circ)$, where $L$ is a vector space and
$\mcdot, \circ$ are two bilinear operations on $L$ satisfying the
following conditions.
\begin{enumerate}
\item[(1)] $(L,\cdot)$ is a commutative associative algebra.
\item[(2)] $(L,\circ)$ is a pre-Lie algebra, that is, Eq.~(\mref{eq:PLie}) holds.
\item[(3)] Eqs.~(\mref{eq:NP1}) and (\mref{eq:NP2}) hold.
\end{enumerate}
\end{defi}

For the relationship among these notions, we have
\begin{lem}\mlabel{lem:PLie}
\begin{enumerate}
\item A Novikov-Poisson algebra is a pre-Lie Poisson algebra;
\mlabel{it:PLie1} \item A PreLie-Com algebra $(L,\cdot,\circ)$
satisfying Eq.~\meqref{eq:NP1} is a pre-Lie Poisson algebra;
\mlabel{it:PLie2} \item A differential Novikov-Poisson algebra is
a PreLie-Com algebra satisfying Eq.~\meqref{eq:NP1} and hence is
pre-Lie Poisson algebra. \mlabel{it:PLie3}
\end{enumerate}
\end{lem}

\begin{proof}
\meqref{it:PLie1}. This follows since a Novikov algebra is a pre-Lie algebra and a
Novikov-Poisson algebra and a pre-Lie Poisson algebra have the same compatibility conditions for the two operations.

\smallskip

\noindent
\meqref{it:PLie2}.
Let $x,y,z\in L$. Then we have
\begin{eqnarray*}
&&y\circ (x\mcdot z)+(x\circ y)\mcdot z-(y\circ x)\mcdot z-x\circ
(y\mcdot z)\\
&&=(y\circ x)\mcdot z+x\mcdot (y\circ z)+(x\circ y)\mcdot z-(y\circ
x)\mcdot z -(x\circ y)\mcdot z-y\mcdot (x\circ z)\\
&&=x\mcdot (y\circ z)-y\mcdot (x\circ z)=(x\mcdot y)\circ z-(y\mcdot
x)\circ z=0.
\end{eqnarray*}
Hence $(L,\cdot,\circ)$ is a pre-Lie Poisson algebra.

\smallskip

\noindent \meqref{it:PLie3}. Note that the identity (\ref{eq:NP2})
in the definition of a differential Novikov-Poisson algebra is
redundant, since by Item.\meqref{it:PLie2},
 it can be derived from
Eqs.~(\ref{eq:PLiecomm}) and (\ref{eq:NP1}). Then the conclusion
follows.
\end{proof}

The commutator of a pre-Lie algebra is a Lie algebra~\mcite{Bu}. Observe that the proof of Theorem~\mref{thm:NPc} only uses the relations in Eqs.~\meqref{eq:NP1} and \meqref{eq:NP2}, which hold for pre-Lie algebras. Thus the same proof as that for Theorem~\mref{thm:NPc} gives

\begin{pro}\mlabel{pro:LPc}
Let $(L,\cdot,\circ)$ be a pre-Lie Poisson algebra. Then
$(L,\cdot,[\;,\;])$ is a \dualp algebra, where the operation $[\;,\;]$ is defined by Eq.~\meqref{eq:com}:
$$[x,y]=x\circ y-y\circ x,\;\; \forall x,y\in L.$$
\end{pro}

\begin{ex}
Let $L$ be a 2-dimensional vector space with a basis $\{e_1,e_2\}$. Then $(L,\circ)$ is a pre-Lie algebra~\mcite{BM0,Bu0} with the non-zero product given by
$$e_1\circ e_1=e_1,\quad e_1\circ e_2=e_2.$$
This pre-Lie algebra is not a Novikov algebra~\mcite{BM}. By a straightforward computation, a commutative (not necessarily associative) algebra $(L,\cdot)$ satisfies
Eq.~(\mref{eq:NP1}) if and only if its non-zero product is given by
$$e_1\mcdot e_1=ae_2,$$
Moreover, with the
above product, $(L,\cdot)$ is associative and Eq.~(\mref{eq:NP2})
is satisfied. Therefore, with the above pre-Lie algebra product
and commutative associative algebra product, $(L,\cdot,\circ)$ is
a pre-Lie Poisson algebra. It is not a Novikov-Poisson algebra.
The induced \dualp algebra $(L,\cdot,[\;,\;])$, where the
operation $[\;,\;]$ is defined by Eq.~(\mref{eq:com}), is
isomorphic to the second \dualp algebra of dimension two with non-abelian Lie product in Section~\mref{ss:ex}.
\end{ex}

Combining Proposition~\mref{pro:LPc} and
Lemma~\mref{lem:PLie}, we have

\begin{cor}
Let $(L,\cdot,\circ)$ be a PreLie-Com algebra. If in addition, Eq.~\meqref{eq:NP1} holds, then $(L,\cdot,[\;,\;])$ is a \dualp
algebra, where the operation $[\;,\;]$ is defined by
Eq.~\meqref{eq:com}: $[x,y]:=x\circ y-y\circ x$ for all
$x,y\in L$. In particular, the conclusion holds when $(L,\cdot,\circ)$ is a differential Novikov-Poisson algebra.
\end{cor}

One of the motivations and intentions of introducing the notion of Novikov-Poisson algebras is so that the tensor product can be taken~\mcite{XuX}. We give the following similar property for pre-Lie Poisson algebras.

\begin{pro} \mlabel{pro:PLietensor} Suppose that $(L_1,\cdot,\circ)$ and $(L_2,\cdot,\circ)$ are
two pre-Lie Poisson algebras. Define two bilinear operations
$\cdot$ and $\circ$ on $L_1\otimes L_2$
respectively by Eq.~\meqref{eq:tensor1} and
\begin{equation}
(x_1\otimes x_2)\circ (y_1\otimes y_2)=x_1\circ y_1\otimes
x_2\cdot y_2+x_1\cdot y_1\otimes x_2\circ y_2,\;\;\forall
x_1,y_1\in L_1, x_2,y_2\in L_2.
\end{equation}
Then $(L_1\otimes L_2,\cdot,\circ)$ is a pre-Lie Poisson algebra.
\end{pro}

\begin{proof}
As is well-known, $(L_1\otimes L_2,\cdot)$ is a commutative associative algebra. Let $x_1,y_1,z_1\in L_1,~x_2,y_2,z_2\in L_2$. Then we
have
\begin{eqnarray*}
&&((x_1\otimes x_2)\circ(y_1\otimes y_2))\circ(z_1\otimes z_2)
-((y_1\otimes y_2)\circ(x_1\otimes x_2))\circ(z_1\otimes z_2)\\
&&-(x_1\otimes x_2)\circ((y_1\otimes y_2)\circ(z_1\otimes z_2))
+(y_1\otimes y_2)\circ((x_1\otimes x_2)\circ(z_1\otimes z_2))\\
&=&((x_1\circ y_1)\circ z_1-(y_1\circ x_1)\circ z_1
-x_1\circ(y_1\circ z_1)+y_1\circ(x_1\circ z_1))\otimes x_2\mcdot
y_2\mcdot z_2\\
&&+((x_1\circ y_1)\mcdot z_1-(y_1\circ x_1)\mcdot
z_1)\otimes(x_\mcdot y_2)\circ z_2 -x_1\mcdot (y_1\circ z_1)\otimes
x_2\circ(y_2\mcdot z_2)\\
&&+y_1\mcdot (x_1\circ z_1)\otimes y_2\circ(x_2\mcdot z_2)+(x_1\mcdot
y_1)\circ z_1\otimes((x_2\circ y_2)\mcdot z_2-(y_2\circ x_2)\mcdot
z_2)\\
&& -x_1\circ(y_1\mcdot z_1)\otimes x_2\mcdot (y_2\circ
z_2)+y_1\circ(x_1\mcdot z_1)\otimes y_2\mcdot
(x_2\circ z_2)\\
&&+x_1\mcdot y_1\mcdot z_1\otimes((x_2\circ y_2)\circ z_2-(y_2\circ
x_2)\circ z_2 -x_2\circ(y_2\circ z_2)+y_2\circ(x_2\circ z_2))\\
&\overset{\textbf{(\mref{eq:PLie})}}{=}&((x_1\circ y_1)\mcdot
z_1-(y_1\circ x_1)\mcdot z_1)\otimes(x_2\mcdot y_2)\circ z_2
-x_1\mcdot (y_1\circ z_1)\otimes x_2\circ(y_2\mcdot z_2)\\
&&+y_1\mcdot (x_1\circ z_1)\otimes y_2\circ(x_2\mcdot z_2)
+(x_1\mcdot y_1)\circ z_1\otimes((x_2\circ y_2)\mcdot z_2-(y_2\circ
x_2)\mcdot z_2) \\
&&-x_1\circ(y_1\mcdot z_1)\otimes x_2\mcdot (y_2\circ
z_2)+y_1\circ(x_1\mcdot z_1)\otimes y_2\mcdot (x_2\circ z_2)\\
&\overset{\textbf{(\mref{eq:NP1})}}{=}&((x_1\circ y_1)\mcdot
z_1-(y_1\circ x_1)\mcdot z_1)\otimes(x_2\mcdot y_2)\circ z_2
-(x_1\mcdot y_1)\circ z_1\otimes(x_2\circ(y_2\mcdot
z_2)-y_2\circ(x_2\mcdot z_2))\\
&&+(x_1\mcdot y_1)\circ z_1\otimes((x_2\circ y_2)\mcdot
z_2-(y_2\circ x_2)\mcdot z_2) -(x_1\circ(y_1\mcdot
z_1)-y_1\circ(x_1\mcdot z_1))\otimes(x_2\mcdot y_2)\circ z_2\\
&=&((x_1\circ y_1)\mcdot z_1-(y_1\circ x_1)\mcdot
z_1-x_1\circ(y_1\mcdot z_1)+y_1\circ(x_1\mcdot z_1))\otimes(x_2\mcdot
y_2)\circ z_2\\
&&+(x_1\mcdot y_1)\circ z_1\otimes((x_2\circ y_2)\mcdot
z_2-(y_2\circ x_2)\mcdot z_2-x_2\circ(y_2\mcdot
z_2)+y_2\circ(x_2\mcdot z_2))\\
&
\overset{\textbf{(\mref{eq:NP2})}}{=}&0.
\end{eqnarray*}
Therefore, $(L_1\otimes L_2,\circ)$ is a pre-Lie algebra. Moreover,
we have
\begin{eqnarray*}
&&((x_1\otimes x_2)\mcdot(y_1\otimes
y_2))\circ(z_1\otimes z_2)\\
 &=&(x_1\mcdot y_1)\circ z_1\otimes
x_2\mcdot y_2\mcdot z_2+x_1\mcdot y_1\mcdot z_1\otimes(x_2\mcdot
y_2)\circ z_2\\
&\overset{\textbf{(\mref{eq:NP1})}}{=}&x_1\mcdot (y_1\circ
z_1)\otimes x_2\mcdot (y_2\mcdot z_2)+x_1\mcdot (y_1\mcdot z_1)\otimes
x_2\mcdot (y_2\circ z_2) \\
&=&(x_1\otimes x_2)\mcdot((y_1\otimes
y_2)\circ(z_1\otimes z_2)).
\end{eqnarray*}
\begin{eqnarray*}
&&((x_1\otimes x_2)\circ(y_1\otimes y_2))\mcdot(z_1\otimes z_2)-
((y_1\otimes y_2)\circ(x_1\otimes x_2))\mcdot(z_1\otimes z_2)\\
&&-(x_1\otimes x_2)\circ((y_1\otimes y_2)\mcdot(z_1\otimes z_2))+
(y_1\otimes y_2)\circ((x_1\otimes x_2)\mcdot(z_1\otimes z_2))\\
&=&((x_1\circ y_1)\mcdot z_1-(y_1\circ x_1)\mcdot
z_1-x_1\circ(y_1\mcdot z_1)+y_1\circ(x_1\mcdot z_1))\otimes x_2\mcdot
y_2\mcdot z_2\\
&&+x_1\mcdot y_1\mcdot z_1\otimes((x_2\circ y_2)\mcdot z_2-(y_2\circ
x_2)\mcdot z_2-x_2\circ(y_2\mcdot z_2)+y_2\circ(x_2\mcdot
z_2))\\
&\overset{\textbf{(\mref{eq:NP2})}}{=}&0.
\end{eqnarray*}
Hence $(L_1\otimes L_2,\cdot,\circ)$ is a pre-Lie
Poisson algebra.
\end{proof}

\section{\Dualp algebras and 3-Lie algebras}
\mlabel{sec:3lie} In this section we present the close
relationship between \dualp algebras and 3-Lie algebras. In one
direction, we show that a \dualp algebra with a derivation gives
rise to a \dualp 3-Lie algebra. A similar construction works for a
\dualp algebra with an involutive endomorphism. In the other direction, the 3-Lie algebra
derived from a strong Poisson algebra
in~\mcite{Dz} is shown to carry a \dualp 3-Lie
algebra structure, instead of a Poisson 3-Lie algebra one.

\subsection{Construction of 3-Lie algebras from \dualp algebras with derivations}

We first recall the notion of a 3-Lie algebra.
\begin{defi}{\rm\mcite{Filippov}}
  A {\bf 3-Lie algebra}  is a vector space $A$ together with a skew-symmetric linear map ($3$-Lie bracket) $[\;,\;,\;]:
\otimes^3 A\rightarrow A$ such that the following {\bf Fundamental Identity (FI)} holds:
\begin{equation}\mlabel{eq:de13Lie}
[[x,y,z],u,v]=[[x,u,v],y,z]+[[y,u,v],z,x]+[[z,u,v],x,y],
\end{equation}
for $x, y, z, u, v\in A$.
\end{defi}
In other words, the adjoint action of $u,v$ in the ternary operation is a derivation.
Given the importance of Poisson algebras and 3-Lie algebras, it is
natural to study the relationship between them. As it turned out, an extra condition is needed for this to work.

\begin{defi}
\mlabel{de:strongP}
{\rm \mcite{Dz}}\quad A Poisson algebra $(L,\cdot,[\;,\;])$ is
called {\bf strong} if
\begin{equation}
[h,x]\mcdot [y,z]+[h,y]\mcdot [z,x]+[h,z]\mcdot [x,y]=0,\;\;\forall
x,y,z,h\in L.
\mlabel{eq:strong}
\end{equation}
\end{defi}

For example, the Poisson algebra $(L,\cdot,[\;,\;])$ in
Example~\mref{ex:Poi} is strong.

We observe that Eq~\meqref{eq:strong} is just Eq.~(\mref{eq:gi4}) in Theorem~\mref{thm:id}, hence holds for all \dualp algebras.
As a 3-Lie variation of Poisson algebras, we have

\begin{defi}\mcite{Dz}\label{defi:3LieP}
 A {\bf Poisson 3-Lie algebra} is a
vector space $L$ with a bilinear operator $\mcdot$ and a ternary
operator $[\;,\;,\;]$ such that $(L,\cdot)$ is a commutative
associative algebra, $(L, [\;,\;,\;])$ is a 3-Lie algebra and the
following equation holds.
\begin{equation}\mlabel{eq:3LiePoi}
[x,y,u\mcdot v]=u\mcdot [x,y,v]+[x,y,u]\mcdot v,\;\;\forall
x,y,u,v\in L.
\end{equation}
Thus the adjoint action of $x,y$ on the commutative associative product is a derivation.

In addition, a Poisson 3-Lie algebra $(L,\cdot,[\;,\;,\;])$ is
called {\bf strong} if the following equation holds.
\small{
\begin{equation}\mlabel{eq:strong-3}
-[x,y,u_1]\mcdot [u_2,u_3,u_4]+[x,y,u_2]\mcdot
[u_1,u_3,u_4]-[x,y,u_3]\mcdot [u_1,u_2,u_4]+[x,y,u_4]\mcdot
[u_1,u_2,u_3]=0,
\end{equation}
}
for all $x,y,u_1,u_2,u_3,u_4\in L$. \mlabel{de:3LPa}
\end{defi}

\begin{pro}{\rm \mcite{Dz}}
Let $(L,\cdot, [\;,\;])$ be a strong Poisson algebra.
Suppose that $D$ is a derivation of $(L,\cdot)$ and
$(L,[\;,\;])$. Define a ternary operation by
\begin{equation}
[x,y,z]:=D(x)\mcdot[y,z]+D(y)\mcdot [z,x]+D(z)\mcdot [x,y], \quad \forall x,y,z\in L.
\mlabel{eq:3-Lie2}
\end{equation}
Then $(L,\cdot,[\;,\;,\;])$ is a strong Poisson 3-Lie algebra.
\mlabel{pro:cons1}
\end{pro}

Let $(L,\cdot)$ be a commutative associative algebra and $D_1,D_2$
be two commutating derivations. Let $(L,[\;,\;])$ be the Lie
algebra defined by Eq.~(\mref{eq:2der}) through $D_1,D_2$ in
Example~\mref{ex:Poi}.
Let $D_3$ be a third derivation of $(L,\cdot)$ commuting with
$D_1$ and $D_2$. It is straightforward to show that $D_3$ is a
derivation of $(L,[\;,\;])$. Hence there is a 3-Lie algebra defined by Eq.~(\mref{eq:3-Lie2}).
Explicitly,
\begin{equation}[x,y,z]:=\det\left( \begin{matrix}  D_1(x) & D_1(y) & D_1(z)\cr D_2(x) & D_2(y) & D_2(z)\cr D_3(x) & D_3(y) &D_3(z)\end{matrix}\right)
=(D_1\wedge D_2\wedge D_3)(x,y,z),\;\;\forall x,y,z\in L.
\end{equation}
This is one of the first and also one of the most important examples of 3-Lie algebras~\mcite{N,T}.

As we will see (Theorem~\mref{thm:3Lie} and Proposition~\mref{pro:3Liepoi}), the strongness condition on the construction of 3-Lie Poisson algebras from Poisson algebras in Eq.~\meqref{eq:3-Lie2} can be removed when the equation is used to construct \dualp 3-Lie algebras from \dualp algebras.

\begin{thm}\mlabel{thm:3Lie} Let $(L,\cdot, [\;,\;])$ be a \dualp algebra and let $D$ be a derivation of $(L,\cdot)$ and $(L,[\;,\;])$. Define a ternary operation on $L$ by
Eq.~\meqref{eq:3-Lie2}:
\begin{equation}
[x,y,z]:=D(x)\mcdot [y,z]+D(y)\mcdot [z,x]+D(z)\mcdot [x,y], \;\;x,y,z,\in
L.\mlabel{eq:3-Lie}
\end{equation}
Then $(L,[\;,\;,\;])$ is a 3-Lie algebra.
\end{thm}

\begin{proof}
Let $x,y,z,u,v\in L$. First we have
\begin{eqnarray}
&&D(x)\mcdot D([y,z])+D(y)\mcdot D([z,x])+D(z)\mcdot D([x,y]) \notag\\
&=&-x\mcdot [D(y),D(z)]-y\mcdot [D(z),D(x)]-z\mcdot
[D(x),D(y)],
\mlabel{eq:identity}
\end{eqnarray}
since by the Leibniz rule of $D$ and Eq.~\meqref{eq:dualp}, both sides equal to the one-half of
$$[D(y),D(x)\mcdot z]+[D(y)\mcdot z, D(x)]+[D(z),D(y)\mcdot x]+[D(z)\mcdot x, D(y)]+[D(z)\mcdot D(x),y]+[D(x)\mcdot y, D(z)].
$$

Furthermore, expanding by Eq.~\eqref{eq:3-Lie} and then applying Eq.~\meqref{eq:identity}, we have
{\small
\begin{eqnarray*}
&&[[x,y,z],u,v]\\
&&=[D(x)\mcdot [y,z],u,v]+[D(y)\mcdot [z,x],u,v]+[D(z)\mcdot
[x,y],u,v]\\
&&=D(D(x)\mcdot [y,z])\mcdot [u,v]+D(u)\mcdot [v,D(x)\mcdot
[y,z]]+D(v)\mcdot [D(x)\mcdot [y,z],u]\\
&&+D(D(y)\mcdot [z,x])\mcdot [u,v]+D(u)\mcdot [v,D(y)\mcdot
[z,x]]+D(v)\mcdot [D(y)\mcdot [z,x],u]\\
&&+D(D(z)\mcdot [x,y])\mcdot [u,v]+D(u)\mcdot [v,D(z)\mcdot
[x,y]]+D(v)\mcdot [D(z)\mcdot [x,y],u]\\
&&=D^2(x)\mcdot [y,z]\mcdot [u,v]+D(x)\mcdot D([y,z])\mcdot
[u,v]+D(u)\mcdot [v,D(x)\mcdot [y,z]]+D(v)\mcdot [D(x)\mcdot
[y,z],u]\\
&&+D^2(y)\mcdot [z,x]\mcdot [u,v]+D(y)\mcdot D([z,x])\mcdot [u,v]+D(u)\mcdot [v,D(y)\mcdot [z,x]]+D(v)\mcdot [D(y)\mcdot [z,x],u]\\
&&+D^2(z)\mcdot [x,y]\mcdot [u,v]+D(z)\mcdot D([x,y])\mcdot
[u,v]+D(u)\mcdot [v,D(z)\mcdot [x,y]]+D(v)\mcdot [D(z)\mcdot
[x,y],u]\\
&&=D^2(x)\mcdot [y,z]\mcdot [u,v]-x\mcdot [D(y),D(z)]\mcdot
[u,v]+D(u)\mcdot [v,D(x)\mcdot [y,z]] -D(v)\mcdot [u,D(x)[y,z]]\\
&&+D^2(y)\mcdot [z,x]\mcdot [u,v]-y\mcdot [D(z),D(x)]\mcdot
[u,v]+D(u)\mcdot [v,D(y)\mcdot [z,x]]-D(v)\mcdot [u,D(y)\mcdot
[z,x]]\\
&&+D^2(z)\mcdot[x,y]\mcdot[u,v]-z\mcdot
[D(x),D(y)]\mcdot[u,v]+D(u)\mcdot[v,D(z)\mcdot[x,y]]-D(v)\mcdot[u,D(z)\mcdot[x,y]].
\end{eqnarray*}
}
Similarly,
{\small
\begin{eqnarray*}
&&[[x,u,v],y,z]\\
&&=D^2(x)\mcdot[u,v]\mcdot[y,z]-x\mcdot [D(u),D(v)]\mcdot
[y,z]+D(y)\mcdot[z,D(x)\mcdot [u,v]]+D(z)\mcdot [D(x)\mcdot [u,v],y]\\
&&+D^2(u)\mcdot[v,x]\mcdot[y,z]-u\mcdot[D(v),D(x)]\mcdot[y,z]+D(y)\mcdot[z,D(u)\mcdot[v,x]]+D(z)\mcdot[D(u)\mcdot
[v,x],y]\\
&&+D^2(v)\mcdot[x,u]\mcdot[y,z]-v\mcdot[D(x),D(u)]\mcdot[y,z]+D(y)\mcdot
[z,D(v)\mcdot[x,u]]+D(z)\mcdot[D(v)\mcdot[x,u],y].
\end{eqnarray*}
\begin{eqnarray*}
&&[[y,u,v],z,x]\\
&&=D^2(y)\mcdot[u,v]\mcdot[z,x]-y\mcdot [D(u),D(v)]\mcdot
[z,x]+D(z)\mcdot [x,D(y)\mcdot[u,v]]+D(x)\mcdot [D(y)\mcdot [u,v],z]\\
&&+D^2(u)\mcdot [v,y]\mcdot [z,x]-u\mcdot [D(v),D(y)]\mcdot
[z,x]+D(z)\mcdot[x,D(u)\mcdot[v,y]]+D(x)\mcdot[D(u)\mcdot[v,y],z]\\
&&+D^2(v)\mcdot [y,u]\mcdot[z,x]-v\mcdot
[D(y),D(u)]\mcdot[z,x]+D(z)\mcdot[x,D(v)\mcdot[y,u]]+D(x)\mcdot[D(v)\mcdot[y,u],z].
\end{eqnarray*}
\begin{eqnarray*}
&&[[z,u,v],x,y]\\
&&=D^2(z)\mcdot [u,v]\mcdot
[x,y]-z\mcdot[D(u),D(v)]\mcdot[x,y]+D(x)\mcdot[y,D(z)\mcdot[u,v]]+D(y)\mcdot[D(z)\mcdot[u,v],x]\\
&&+D^2(u)\mcdot[v,z]\mcdot [x,y]-u\mcdot[D(v),D(z)]\mcdot[x,y]+D(x)\mcdot[y,D(u)\mcdot[v,z]]+D(y)\mcdot[D(u)\mcdot[v,z],x]\\
&&+D^2(v)\mcdot[z,u]\mcdot
[x,y]-v\mcdot[D(z),D(u)]\mcdot[x,y]+D(x)\mcdot[y,D(v)\mcdot[z,u]]+D(y)\mcdot[D(v)\mcdot[z,u],x].
\end{eqnarray*}
}
Therefore we rewrite
$$[[x,u,v],y,z]+[[y,u,v],z,x]+[[z,u,v],x,y]=\sum_{i=1}^8 (A_i),$$
where

\begin{eqnarray*}
(A_1)&:=&D^2(x)\mcdot[u,v]\mcdot[y,z]+D^2(y)\mcdot[u,v]\mcdot[z,x]+D^2(z)\mcdot
[u,v]\mcdot [x,y],\\
(A_2)&:=&D^2(u)\mcdot[v,x]\mcdot[y,z]+D^2(u)\mcdot [v,y]\mcdot
[z,x]+D^2(u)\mcdot[v,z]\mcdot
[x,y]\\&&+D^2(v)\mcdot[x,u]\mcdot[y,z]+D^2(v)\mcdot
[y,u]\mcdot[z,x]+D^2(v)\mcdot[z,u]\mcdot [x,y],\\
(A_3)&:=&-x\mcdot [D(u),D(v)]\mcdot [y,z]-y\mcdot [D(u),D(v)]\mcdot
[z,x]-z\mcdot[D(u),D(v)]\mcdot[x,y],\\
(A_4)&:=&-u\mcdot[D(v),D(x)]\mcdot[y,z]-u\mcdot [D(v),D(y)]\mcdot
[z,x]-u\mcdot[D(v),D(z)]\mcdot[x,y],\\
(A_5)&:=&-v\mcdot[D(x),D(u)]\mcdot[y,z]-v\mcdot
[D(y),D(u)]\mcdot[z,x]-v\mcdot[D(z),D(u)]\mcdot[x,y],\\
(A_6)&:=& D(x)\mcdot[D(u)\mcdot[v,y],z]+D(x)\mcdot[y,D(u)\mcdot[v,z]]+
D(y)\mcdot[z,D(u)\mcdot[v,x]]\\
&&+D(y)\mcdot[D(u)\mcdot[v,z],x]+D(z)\mcdot[D(u)\mcdot
[v,x],y]+D(z)\mcdot[x,D(u)\mcdot[v,y]]\\
(A_7)&:=&D(x)\mcdot[D(v)\mcdot[y,u],z]+D(x)\mcdot[y,D(v)\mcdot[z,u]]+
D(y)\mcdot [z,D(v)\mcdot[x,u]]\\
&&+D(y)\mcdot[D(v)\mcdot[z,u],x]+D(z)\mcdot[D(v)\mcdot[x,u],y]+D(z)\mcdot[x,D(v)\mcdot[y,u]]\\
(A_8)&:=&(D(x)\mcdot[y,D(z)\mcdot[u,v]]+D(z)\mcdot
[D(x)\mcdot
[u,v],y])+(D(y)\mcdot[z,D(x)\mcdot [u,v]]\\
&&+D(x)\mcdot [D(y)\mcdot [u,v],z])+(D(z)\mcdot
[x,D(y)\mcdot[u,v]]+D(y)\mcdot[D(z)\mcdot[u,v],x]).
\end{eqnarray*}

By Eq.~(\mref{eq:gi4}), we have $(A_2)=0$ and by
Eq.~(\mref{eq:gi1}), we have $(A_3)=0$.

By Eq.~(\mref{eq:gi2}), we have
\begin{eqnarray*}
(A6)&=&D(x)\mcdot [v,D(u)\mcdot [y,z]]+D(y)\mcdot [v,D(u)\mcdot
[z,x]]+D(z)\mcdot [v,D(u)\mcdot [x,y]]\\
 (A7)&=&-D(x)\mcdot [u,D(v)\mcdot [y,z]]-D(y)\mcdot [u,D(v)\mcdot [z,x]]-D(z)\mcdot [u,D(v)\mcdot
 [x,y]].
\end{eqnarray*}

Applying Eq.~\meqref{eq:gi2} and then Eq.~\meqref{eq:gi6}, we obtain
$$-u[v,x][y,z]+x[v[y,u],z]+x[y,v[z,u]] = -u[v,x][y,z]-x[v[z,y],u]= -v[u,x[y,z]].
$$
Replacing $x$ and $v$ by $D(x)$ and $D(v)$ respectively, this gives

\begin{equation}
-u[D(v),D(x)][y,z] +D(x)[D(v)[y,u],z] + D(x)[y,D(v)[z,u]]
=-D(v)[u,D(x)[y,z]].
\mlabel{eq:imp}
\end{equation}

From Eq.~(\mref{eq:imp}), we obtain
\begin{eqnarray*}
(A_4)+(A_7)&=&-D(v)\mcdot [u,D(x)[y,z]]-D(v)\mcdot [u,D(y)\mcdot
[z,x]]-D(v)\mcdot[u,D(z)\mcdot[x,y]],\\
(A_5)+(A_6)&=&D(u)\mcdot [v,D(x)\mcdot [y,z]]+D(u)\mcdot [v,D(y)\mcdot
[z,x]]+D(u)\mcdot[v,D(z)\mcdot[x,y]],
\end{eqnarray*}
and from Eq.~\meqref{eq:gi6}, we obtain
$$
(A_8)=-x\mcdot [D(y),D(z)]\mcdot [u,v]-y\mcdot [D(z),D(x)]\mcdot
[u,v]-z\mcdot [D(x),D(y)]\mcdot[u,v].
$$
In summary,
$$[[x,y,z],u,v]=(A_1)+(A_4)+(A_5)+(A_6)+(A_7)+(A_8)=\sum_{i=1}^8(A_i).$$
Therefore $(L,[\;,\;,\;])$ is a 3-Lie algebra.
\end{proof}

As a direct consequence of Theorem~\mref{thm:3Lie}, we obtain the following construction of 3-Lie algebras in~\mcite{Dz} (see also Example~\mref{ex:Dz}).

\begin{cor}
Let $(L,\cdot)$ be a commutative associative algebra and $D_1,D_2$ be two commutating derivations. Then there exists a 3-Lie algebra defined by
\begin{eqnarray}
[x,y,z]&:=&D_2(x)\mcdot (y\mcdot D_1(z)-z\mcdot D_1(x))+D_2(y)\mcdot (z\mcdot D_1(x)-x\mcdot D_1(z))\nonumber\\
&\mbox{}&+D_2(z)\mcdot(x\mcdot D_1(y)-y\mcdot D_1(x)),\;\;\forall x,y,z\in L.\mlabel{eq:3Lie-CA}
\end{eqnarray}
\mlabel{co:3Lie-CA}
\end{cor}

\begin{proof}
By Corollary~\mref{co:dualpSG}, $(L,\cdot)$ together with the bracket defined by
$$ [x,y]:=x\mcdot D_1(y)-y\mcdot D_1(x), \quad \forall x, y\in L,$$
gives a \dualp algebra. A direct checking shows that $D_2$ is a derivation on $(L,[\;,\;])$. Thus by Theorem~\mref{thm:3Lie}, Eq.~\meqref{eq:3Lie-CA} gives a 3-Lie algebra structure on $L$.
\end{proof}

We can rewrite Eq.~(\mref{eq:3Lie-CA}) as
\begin{equation}[x,y,z]:=\det\left( \begin{matrix}  x & y & z\cr D_1(x) & D_1(y) & D_1(z)\cr D_2(x) & D_2(y) & D_2(z)\cr\end{matrix}\right)=({\rm id}\wedge D_1\wedge D_2)(x,y,z),\;\;\forall x,y,z\in L.\mlabel{eq:3-Liefrom}\end{equation}

In fact, the 3-Lie algebras given by Theorem~\mref{thm:3Lie} satisfies an additional equation.

\begin{pro}
Let $(L,\cdot, [\;,\;])$ be a \dualp
algebra. Suppose that $D$ is a derivation of $(L,\cdot)$ and
$(L,[\;,\;])$. Then for the 3-Lie algebra defined by Eq.~\meqref{eq:3-Lie}, the triple $(L,\cdot,[\;,\;,\;])$ in a {\bf \dualp 3-Lie algebra} in the sense that the following additional equation holds.
\begin{equation}
3u\mcdot [x,y,z]=[x\mcdot u,y,z]+[x,y\mcdot u,z]+[x,y,z\mcdot u], \;\;\forall ~x,y,z,u\in L.\mlabel{eq:dualp3Lie}
\end{equation}
\mlabel{pro:3Liepoi}
\end{pro}

\begin{proof} Let $ x,y,z,u\in L$. Then we have
\begin{eqnarray*}
&&[x\mcdot u,y,z]+[x,y\mcdot u,z]+[x,y,z\mcdot u]\\
&&=D(x\mcdot u)\mcdot [y,z]+D(y)\mcdot [z,x\mcdot u]+D(z)\mcdot [x\mcdot u,y]+D(x)\mcdot [y\mcdot u,z]+D(y\mcdot u)\mcdot [z,x]\\
&&+D(z)\mcdot [x,y\mcdot u]++D(x)\mcdot [y,z\mcdot u]+D(y)\mcdot [z\mcdot u,x]+D(z\mcdot u)\mcdot[x,y]\\
&&=3u\mcdot (D(x)\mcdot [y,z]+D(y)\mcdot [z,x]+D(z)\mcdot [x,y])+D(u)\mcdot (x\mcdot [y,z]+y\mcdot [z,x]+z\mcdot [x,y])\\
&&=3u\mcdot [x,y,z].
\end{eqnarray*}
Hence the conclusion holds.
\end{proof}

In \mcite{Dz}, there is also a notion of a Poisson $n$-Lie algebra
which consists of a commutative associative algebra and an $n$-Lie
algebra, but the compatibility condition is that the adjoint
operators of the $n$-Lie algebra are the derivations of the
commutative associative algebra, thus generalizing the notions
of a Poisson algebra and a Poisson 3-Lie algebra.

For \dualp algebras, we introduce

\begin{defi}
Let $n\geq 2$ be an integer. A {\bf \dualp $n$-Lie algebra} is a triple $(L,\cdot,\mu)$ where $(L,\cdot)$ is a commutative associative algebra and $(L,\mu)$ is an $n$-Lie algebra satisfying the following condition.
$$ n w \mu(x_1,\cdots,x_n)=\sum_{i=1}^n \mu(x_1,\cdots,wx_i,\cdots,x_n), \quad \forall w, x_1,\cdots, x_n\in L.$$
\mlabel{de:dualpn}
\end{defi}

It is reasonable to expect that an \dualp $(n+1)$-algebra can be
obtained from an \dualp $n$-Lie algebra with a derivation,
generalizing Proposition~\mref{pro:3Liepoi}. To be precise, we propose

\begin{conj}
Let $n\geq 2$ be an integer. Let $(L,\cdot,\mu_n)$ be a \dualp $n$-Lie algebra and let $D$ be a derivation of $(L,\cdot)$ and $(L,\mu_n)$.
Define an
$(n+1)$-ary operation
$$\mu_{n+1}(x_1,\cdots,x_{n+1}):=\sum_{i=1}^{n+1} D(x_i) \mu_n(x_1,\cdots, \check{x}_i,\cdots,x_{n+1}),\quad \forall x_1,\cdots,x_{n+1}\in L,
$$
where $\check{x}_i$ means that the $i$-th entry is omitted. Then $(L,\cdot,\mu_{n+1})$ is a \dualp $(n+1)$-Lie algebra. \mlabel{cj:dualp}
\end{conj}

This conjecture gives a natural interpretation of the scalar $2$ in the
compatibility condition of a \dualp algebra: it is simply
the arity of the operation of the Lie algebra.

\subsection{Another construction of 3-Lie algebras from \dualp algebras}

There is also a construction of 3-Lie algebras from \dualp algebras with certain linear transformations instead of the derivations.

\begin{thm}\mlabel{thm:const3}
Let $(L,\cdot,[\;,\;])$ be a \dualp algebra. Let $f$ be an endomorphism of $(L,\cdot)$ satisfying $f^2=\id$, that is $f$ is
an involutive endomorphism of $(L,\cdot)$ and
\begin{equation} f([x,y])=-[f(x),f(y)], \;\;\forall x,y\in L.\mlabel{eq:anti}
\end{equation}
Define a ternary operation on $L$ by
\begin{equation}\mlabel{eq:const3-de}
[x,y,z]:=f(x)[y,z]+f(y)[z,x]+f(z)[x,y],\;\;\forall x,y,z\in L.
\end{equation}
Then $(L,[\;,\;,\;])$ is a $3$-Lie algebra.
\end{thm}
Alternatively, in replacing $f$ by $-f$, we can also assume that
$f$ is an involutive endomorphism of $(L,[\;,\;])$ and  satisfies $f(x\mcdot y)=-f(x)\mcdot f(y)$.

\begin{proof} Let $x,y,z,u,v\in L$. Applying Eqs.~\meqref{eq:anti} and \meqref{eq:const3-de}, we obtain
$$[[x,u,v],y,z]+[[y,u,v],z,x]+[[z,u,v],x,y]-[[x,y,z],u,v]=\sum_{i=1}^{10} (A_i),$$
where
\begin{eqnarray*}
(A_1)&:=& -x\mcdot [f(u),f(v)]\mcdot [y,z] -y\mcdot [f(u),f(v)]\mcdot [z,x]-z\mcdot [f(u),f(v)]\mcdot [x,y],\\
(A_2)&:=& x\mcdot [f(y),f(z)]\mcdot [u,v]+f(y)\mcdot [f(z)\mcdot [u,v],x]+f(z)\mcdot [x,f(y)\mcdot [u,v]],\\
(A_3)&:=& y\mcdot [f(z),f(x)]\mcdot [u,v]+f(z)\mcdot [f(x)\mcdot [u,v],y]+f(x)\mcdot [y,f(z)\mcdot [u,v]],\\
(A_4)&:=& z\mcdot [f(x),f(y)]\mcdot [u,v]+f(x)\mcdot [f(y)\mcdot [u,v],z]+f(y)\mcdot [z,f(x)\mcdot [u,v]],\\
(A_5)&:=& -f(u)\mcdot [v,f(x)\mcdot [y,z]]-v\mcdot[f(x),f(u)]\mcdot [y,z]+f(x)\mcdot [f(u)[v,y],z]\\
&&+f(x)\mcdot [y,f(u)\mcdot [v,z]],\\
(A_6)&:=& -f(u)\mcdot [v,f(y)\mcdot [z,x]]-v\mcdot[f(y),f(u)]\mcdot [z,x]+f(y)\mcdot [f(u)[v,z],x]\\
&&+f(y)\mcdot [z,f(u)\mcdot [v,x]],\\
(A_7)&:=& -f(u)\mcdot [v,f(z)\mcdot [x,y]]-v\mcdot[f(z),f(u)]\mcdot [x,y]+f(z)\mcdot [f(u)[v,x],y]\\
&&+f(z)\mcdot [x,f(u)\mcdot [v,y]],\\
(A_8)&:=& f(v)\mcdot [f(x)\mcdot [y,z],u]-u\mcdot[f(v),f(x)]\mcdot [y,z]+f(x)\mcdot [f(v)[y,u],z]\\
&&+f(x)\mcdot [y, f(v)\mcdot [z,u]],\\
(A_9)&:=& f(v)\mcdot [f(y)\mcdot [z,x],u]-u\mcdot[f(v),f(y)]\mcdot [z,x]+f(y)\mcdot [f(v)[z,u],x]\\
&&+f(y)\mcdot [z, f(v)\mcdot [x,u]],\\
(A_{10})&:=& f(v)\mcdot [f(z)\mcdot [x,y],u]-u\mcdot[f(v),f(z)]\mcdot [x,y]+f(z)\mcdot [f(v)[x,u],y]\\
&&+f(z)\mcdot [x, f(v)\mcdot [y,u]].
\end{eqnarray*}

By Eq.~(\mref{eq:gi1}), we have $(A_1)=0$.

By Eqs.~(\mref{eq:dualp}) and ~(\mref{eq:gi1}), we have
\begin{eqnarray*}
x\mcdot [f(y),f(z)]\mcdot [u,v]&=& x\mcdot ([f(y)\mcdot [u,v],f(z)]+[f(y),f(z)\mcdot [u,v]])\\&&+(f(y)\mcdot [f(z),x]+f(z)\mcdot [x,f(y)])\mcdot [u,v].
\end{eqnarray*}
By Eq.~(\mref{eq:gi1}) again, we have
\begin{eqnarray*}
(A_2) &=& f(z)\mcdot [x,f(y)\mcdot [u,v]]+f(y)\mcdot [f(z)\mcdot [u,v],x]+ x\mcdot ([f(y)\mcdot [u,v],f(z)]\\
&&+[f(y),f(z)\mcdot [u,v]])+(f(y)\mcdot [f(z),x]+f(z)\mcdot [x,f(y)])\mcdot [u,v]\\
&=&0.
\end{eqnarray*}
Similarly, we have $(A_3)=(A_4)=0$.

By Eq.~(\mref{eq:gi2}), we have
\begin{eqnarray*}
(A_5)&=& -f(u)\mcdot [v,f(x)\mcdot [y,z]]-v\mcdot[f(x),f(u)]\mcdot [y,z]-f(x)\mcdot [f(u)\mcdot [y,z],v].
\end{eqnarray*}
Therefore due to the identity $(A_2)=0$, we show that $(A_5)=0$. Similarly, $(A_6)=(A_7)=(A_8)=(A_9)=(A_{10})=0$.
Hence $(L,[\;,\;,\;])$ is a $3$-Lie algebra.
\end{proof}

\begin{rmk}
In general, under the same condition in Theorem~\mref{thm:const3},
$(L,\cdot,[\;,\;,\;])$ is not a \dualp 3-Lie algebra. In fact, it is a \dualp 3-Lie algebra if and only $f$ satisfies an additional condition
$$(f(u)-u)\mcdot(f(x)\mcdot [y,z]+f(y)\mcdot [z,x]+f(z)\mcdot [x,y])=0,\;\;\forall u,x,y,z\in L.$$
\end{rmk}

\begin{cor}
Let $(L,\cdot)$ be a commutative associative algebra and $D$ be
derivation. Let $(L,[\;,\;])$ be the Lie algebra defined by
Eq.~(\mref{eq:SG}). If $f$ is an involutive endomorphism and
$Df=-fD$, then Eq.~(\mref{eq:anti}) holds. Hence there is a 3-Lie
algebra $(L,[\;,\;,\;])$ defined by
Eq.~(\mref{eq:const3-de}). Moreover,
\begin{eqnarray*}
[x,y,z]&=&f(x)\mcdot [y,z]+f(y)\mcdot [z,x]+f(x)\mcdot [y,z]\\
&=&\det \left(\begin{matrix} D(x) &D(y) &D(z)\cr f(x) &f(y) &f(z)\cr x&y&z\cr\end{matrix}\right)\\
&=&(D\wedge f\wedge {\rm Id})(x,y,z),\quad \forall x,y,z\in L.
\end{eqnarray*}
\end{cor}

\begin{proof} For all $x,y\in L$, we have
$$f([x,y])=f(x\mcdot D(y)-y\mcdot D(x))=f(x)\mcdot fD(y)-f(y)\mcdot fD(x)=-[f(x),f(y)].$$
Hence by Theorem~\mref{thm:const3}, $(L,[\;,\;,\;])$ defined by Eq.~(\mref{eq:const3-de}) is a 3-Lie algebra.
The other conclusion follows immediately.
\end{proof}

This construction of 3-Lie
algebras recovers \cite[Theorem
3.3]{BW} which was obtained by a direct verification.

\subsection{From Poisson algebras to \dualp 3-Lie algebras}
\mlabel{ss:poisto3dualp}

In addition to Proposition~\mref{pro:cons1}, there is another construction of 3-Lie algebras from Poisson algebras.
\begin{pro}{\rm \mcite{Dz}}
Let $(L,\cdot, [\;,\;])$ be a strong Poisson algebra.
Define
\begin{equation}
[x,y,z]=x\mcdot [y,z]+y\mcdot [z,x]+z\mcdot [x,y],\;\;\forall
x,y,z\in L.\mlabel{eq:Poi-id}
\end{equation}
Then $(L,[\;,\;,\;])$ is a 3-Lie algebra satisfying
Eq.~\meqref{eq:strong-3}. \mlabel{it:cons2}
\mlabel{pro:cons2}
\end{pro}

Let $(L,\cdot,[\;,\;])$ be a Poisson algebra. Suppose that the ternary operation $[\;,\;,\;]$ defined by Eq.~\meqref{eq:Poi-id} gives a 3-Lie algebra $(L,[\;,\;,\;])$. A natural question is whether the triple $(L,\cdot,[\;,\;,\;])$ is a Poisson 3-Lie algebra.
We show that the triple is in fact a \dualp 3-Lie algebra. Indeed, as we will see in Remark~\mref{rmk:trivial-3}, the triple is a Poisson 3-Lie algebra only under strict conditions.

\begin{thm} Let $(L,\cdot,[\;,\;])$ be a Poisson algebra. If Eq.~\meqref{eq:Poi-id} defines a 3-Lie algebra $(L,[\;,\;,\;])$,
then  $(L,\cdot, [\;,\;,\;])$ is a \dualp 3-Lie algebra. In particular, if $(L,\cdot, [\;,\;])$ is a strong Poisson algebra, then the 3-Lie algebra $(L,[\;,\;,\;])$ obtained in Proposition~\mref{pro:cons2} gives a \dualp 3-Lie algebra $(L,\cdot, [\;,\;,\;])$.
\mlabel{thm:Poi-dualp}
\end{thm}

\begin{proof} Let $u,x,y,z\in L$. Then we have
\begin{eqnarray*}
&&[u\mcdot x,y,z]=u\mcdot x\mcdot [y,z]+y\mcdot [z,u\mcdot x]+z\mcdot
[u\mcdot x,y]\\&&\mbox{}\hspace{1.8cm}= u\mcdot x\mcdot [y,z]+y\mcdot
u\mcdot [z,x]+y\mcdot x\mcdot [z,u]+z\mcdot u\mcdot [x,y]+z\mcdot x\mcdot
[u,y],\\&& [x,u\mcdot y,z]=x\mcdot u\mcdot [y,z]+x\mcdot y\mcdot
[u,z]+u\mcdot
y\mcdot [z,x]+z\mcdot u\mcdot [x,y]+z\mcdot y\mcdot [x,u],\\
&&[x,y,u\mcdot z]=x\mcdot u\mcdot [y,z]+x\mcdot z\mcdot [y,u]+y\mcdot
u\mcdot [z,x]+y\mcdot z\mcdot [u,x]+u\mcdot z\mcdot [x,y].
\end{eqnarray*}
Taking the sum of the three identities above, we get
\begin{eqnarray*}
[u\mcdot x,y,z]+[x,u\mcdot y,z]+[x,y,u\mcdot z]&=&3u\mcdot (x\mcdot
[y,z]+y\mcdot [z,x]+z\mcdot [x,y])\\
&=&3u\mcdot [x,y,z].
\end{eqnarray*}
Hence $(L,\cdot, [\;,\;,\;])$ is a \dualp 3-Lie algebra.
\end{proof}

\begin{ex}\mlabel{ex:Dz}
Here is an interesting instance where we can compare the construction of a 3-Lie algebra from a strong Poisson algebra by Proposition~\mref{pro:cons2} (and Theorem~\mref{thm:Poi-dualp}) with the construction of a 3-Lie algebra from a \dualp algebra by Corollary~\mref{co:3Lie-CA}.

Let $(L,\cdot)$ be a commutative associative algebra and $D_1,D_2$
be two commutating derivations. On the one hand, the derivations
$D_1,D_2$ on $(L,\cdot)$ gives a Lie algebra $(L,[\;,\;])$ by
Eq.~\meqref{eq:2der} in Example~\mref{ex:Poi}, turning
$(L,\cdot,[\;,\;])$ into a strong Poisson algebra (see the remark
after Definition~\mref{de:strongP}), which in turn gives rise to
the 3-Lie algebra in Eq.~(\mref{eq:Poi-id}). On the other hand,
the derivation $D_1$ on $(L,\cdot)$ gives a \dualp algebra
$(L,\cdot,[\;,\;])$ by Corollary~\mref{co:dualpSG}, which then
gives rise to a \dualp 3-Lie algebra by
Eq.~(\mref{eq:3-Liefrom}). The two constructions of 3-Lie
algebras coincide since in either case, the 3-Lie algebra is
defined by
$$[x,y,z]={\rm Id}\wedge (D_1\wedge D_2)(x,y,z)=({\rm Id}\wedge D_1)\wedge D_2(x,y,z),\;\;\forall x,y,z\in L.$$
\end{ex}

\begin{rmk} Let $(L,\cdot,[\;,])$ be a Poisson algebra. If the strongness condition Eq.~(\mref{eq:gi4}) is replaced by
Eq.~(\mref{eq:gi2}), that is, the following equation holds,
\begin{equation}
[h\mcdot [x,y],z]+[h\mcdot [y,z],x]+[h\mcdot [z,x],y]=0, \;\;\forall
h,x,y,z\in L,
\end{equation}
then Eq.~(\mref{eq:Poi-id}) still defines a 3-Lie algebra by a straightforward proof.

Moreover,  in this case, the triple $(L,\cdot, [\;,\;,\;])$ is also a
\dualp 3-Lie algebra. One can show that for a Poisson  algebra,
Eq.~(\mref{eq:gi2}) follows from  Eq.~(\mref{eq:gi4}). Note that
for a \dualp algebra, both Eqs.~(\mref{eq:gi2}) and
(\mref{eq:gi4}) hold automatically.
\end{rmk}

Similar to Proposition~\mref{pp:inter0}, we have

\begin{pro}\mlabel{pro:trivial-3}
 Let $(L,\cdot)$ be a commutative associative algebra and $(L,[\;,\;,\;])$ be a 3-Lie algebra. Then $(L,\cdot,[\;,\;,\;])$ is both a Poisson 3-Lie
and a \dualp 3-Lie algebra if and only if
\begin{equation}
u\mcdot [x,y,z]=[u\mcdot x,y,z]=0,\;\;\forall u,x,y,z\in L.
\mlabel{eq:mix3}
\end{equation}
\end{pro}

\begin{proof}
($\Longrightarrow$) Let $u,x,y,z\in L$. Substituting
Eq.~(\mref{eq:3LiePoi}) into Eq.~(\mref{eq:dualp3Lie}), we have
$$x\mcdot [u,y,z]+y\mcdot [u,z,x]+z\mcdot [u,x,y]=0.$$
Substituting Eq.~(\mref{eq:dualp3Lie}) again into the above
equation, we have
$$[u\mcdot x, y,z]+[x,u\mcdot y,z]+[x,y,u\mcdot z]=0.$$
The left hand side of the above is exactly $3 u\mcdot [x,y,z]$.
Therefore we have $u\mcdot [x,y,z]=0$. Hence $[u\mcdot x,y,z]=0$.

($\Longleftarrow$) It is obvious.
\end{proof}

Putting Theorem~\mref{thm:Poi-dualp} and Proposition~\mref{pro:trivial-3} together, we obtain

\begin{rmk}
Let $(L,\cdot,[\;,\;])$ be a Poisson algebra and let  $[\;,\;,\;]$
be the ternary operation defined in Eq.~\meqref{eq:Poi-id}. Then
the triple $(L,\cdot,[\;,\;,\;])$ is a \dualp algebra by
Theorem~\mref{thm:Poi-dualp}. If the triple is also a Poisson
3-Lie algebra, then according to Proposition~\mref{pro:trivial-3},
it is necessary that Eq.~\meqref{eq:mix3} holds, that is, all the
mixed products of $\cdot$ and $[\;,\;,\;]$
are trivial. In particular, if $(L,\cdot, [\;,\;])$ is a strong
Poisson algebra, then the 3-Lie algebra $(L,[\;,\;,\;])$ defined
by Eq.~\meqref{eq:Poi-id} gives rise to a Poisson 3-Lie algebra $(L,\cdot,[\;,\;,\;])$ only if Eq.~\meqref{eq:mix3} holds.
\mlabel{rmk:trivial-3}
\end{rmk}

\medskip

\noindent
{\bf Acknowledgements.}  This work is supported by
 National Natural Science Foundation of China (Grant Nos. 11771190 and 11931009).   C. Bai is also
supported by the Fundamental Research Funds for the Central
Universities and Nankai ZhiDe Foundation.

\end{document}